\documentclass{amsart}

\usepackage{color}

\usepackage{amsxtra,amscd}
\usepackage{mathtools}
\usepackage{nicefrac}
\usepackage{leftidx}
\mathtoolsset{showonlyrefs} 
\usepackage{ mathrsfs }

\usepackage{url}

\numberwithin{equation}{section}

\theoremstyle{definition}
\newtheorem{defi}{Definition}[section]
\newtheorem{ejem}{Example}

\theoremstyle{plain}
\newtheorem{prop}[defi]{Proposition}
\newtheorem{teo}[defi]{Theorem}
\newtheorem{cor}[defi]{Corollary}
\newtheorem{lem}[defi]{Lemma}

\theoremstyle{remark}
\newtheorem{remark}[defi]{Remark}

\newcommand{\lp}{{L}^{p}(\Omega)}
\newcommand{\lpp}{{L}^{\scriptscriptstyle p}(\Omega)}
\newcommand{\wsp}{	{W}^{s,p}(\Omega)}
\newcommand{\wspp}{{W}^{\scriptscriptstyle s,p}(\Omega)}
\newcommand{\wwsp}{\mathcal{W}^{s,p}(\Omega)}
\newcommand{\wwspr}{\mathcal{W}^{s,p}(\mathbb{R}^{n+m})}
\newcommand{\wwspp}{\mathcal{W}^{\scriptscriptstyle s,p}(\Omega)}

\DeclareMathOperator*{\supp}{supp}
\DeclareMathOperator{\diam}{diam}

\title[Eigenvalues for a nonlocal pseudo $p-$Laplacian]{Eigenvalues for a nonlocal pseudo $p-$Laplacian}
\author[L. M. Del Pezzo and J. D. Rossi]
{Leandro M. Del Pezzo and Julio D. Rossi}

\address{Leandro M. Del Pezzo and Julio D. Rossi \hfill\break\indent
CONICET and Departamento  de Matem{\'a}tica, FCEyN, Universidad de Buenos Aires,
\hfill\break\indent Pabellon I, Ciudad Universitaria (1428),
Buenos Aires, Argentina.}

\email{{\tt ldpezzo@dm.uba.ar}, {\tt jrossi@dm.uba.ar}}

\keywords{fractional $p-$Laplacian,  eigenvalues, Dirichlet boundary
conditions. \\
\indent 2010 {\it Mathematics Subject Classification: } 35P30,  35J92, 35R11. 
}
\thanks{
Leandro M. Del Pezzo was partially supported by
CONICET PIP 5478/1438  (Argentina) and Julio D. Rossi  
by MTM2011-27998, (Spain)}

\begin{document}

\begin{abstract} 
	In this paper we study the eigenvalue problems for a nonlocal operator 
	of order $s$ that is analogous to the local pseudo $p-$Laplacian. 
	We show that there is a sequence of eigenvalues $\lambda_n \to \infty$ 
	and that the first one is positive, simple, isolated  and has a positive and 
	bounded associated eigenfunction.  
	For the first eigenvalue we also analyze the limits as $p\to \infty$ 
	(obtaining a limit nonlocal eigenvalue problem analogous to the
	pseudo infinity Laplacian) and as $s\to 1^-$ 
	(obtaining the first eigenvalue for a local operator of $p-$Laplacian type). 
	To perform this study we have to introduce anisotropic fractional Sobolev spaces
	and prove some of their properties.
\end{abstract}

\maketitle

\section{Introduction}\label{intro}

	Our main goal is to introduce a nonlocal operator that is a nonlocal 
	analogous to the local pseudo $p-$Laplacian, 
	$\Delta_{p,x} u + \Delta_{p,y}u$ 
	(here the subindexes $x$ and $y$ denote differentiation with respect to the 
	$x \in {\mathbb{R}}^n $ and $y\in {\mathbb{R}}^m$ variables respectively). The local pseudo $p-$Laplacian appears 
	naturally when one considers critical points of the functional 
	$F(u) = \int_{\Omega} |\nabla_x u|^p + |\nabla_y u|^p \, dxdy$. 
	See \cite{BK,Dp,J,M,RS}.
	On the other hand, recently, it was introduced a nonlocal 
	$p-$Laplacian that is given by
	\[
		(-\Delta)_p^s v(x)= 2\text{ P.V.}
		\int_{\mathbb{R}^k} \dfrac{|v(x)-v(y)|^{p-2}(v(x)-v(y))}{|x-y|^{k+ps}}
		\, dx,
	\]
	the symbol P.V. stands for the principal value of the integral. We will omit it in what follows.
	For references involving this kind of operator we refer to 
	\cite{Brasco,DKP,DRV,FP,IMS,Jilka,LL,Bisci1,Bisci2, Bisci3} and 
	references therein.
	
	Here, we introduce the following nonlocal operator that we will call the 
	nonlocal pseudo $p-$Laplacian,
	\begin{align*}
		\mathcal{L}_{s,p}(u)(x,y)&\coloneqq
		2 \int_{\mathbb{R}^n}
		\dfrac{|u(x,y)-u(z,y)|^{p-2}(u(x,y)-u(z,y)
		)}{|x-z|^{n+sp}}dz\\
		&\qquad +2 \int_{\mathbb{R}^m}
		\dfrac{|u(x,y)-u(x,w)|^{p-2}(u(x,y)-u(x,w)
		)}{|y-w|^{m+sp}}dw.
	\end{align*}

	The natural space to consider when one deals with the operator 
	$\mathcal{L}_{s,p}$ is given by
	\[
		\wwspr\coloneqq
		\left\{u\in L^p(\mathbb{R}^{n+m})\colon
		[u]_{\wwspr}^p<\infty
		\right\},
	\]
	where for $p<+\infty,$
	\begin{align*}
		[u]_{\wwspr}^p
				\coloneqq\int_{\mathbb{R}^{n+m}}\int_{\mathbb{R}^{n}}&
		\dfrac{|u(x,y)-u(z,y)|^p}{|x-z|^{n+sp}} dzdxdy\\
		&+
		\int_{\mathbb{R}^{n+m}}\int_{\mathbb{R}^m}
		\dfrac{|u(x,y)-u(x,w)|^p}{|y-w|^{m+sp}} dwdxdy 	
	\end{align*}
	 and for $p=+\infty$,
	\begin{align*}
		[u]_{\mathcal{W}^{s,\infty}(\mathbb{R}^{n+m})}
		&\coloneqq\max\left\{\sup\left\{
		\dfrac{|u(x,y)-u(z,y)|}{|x-z|^{s}}\colon (x,y)\neq(z,y) 
		\right\}\right.;\\
		&\qquad\qquad\left.\sup\left\{
		\dfrac{|u(x,y)-u(x,w)|}{|y-w|^{s}}\colon (x,y)\neq(x,w)\right\}
		\right\}.
	\end{align*}

	In this paper, we deal with the eigenvalue problem for this operator, 
	that is, given a bounded domain $\Omega$ we look for pairs $(\lambda, u)$
	such that $\lambda \in {\mathbb{R}}$ and
	$u\in \widetilde{\mathcal{W}}^{s,p}(\Omega)\setminus\{0\}$ are such
	that $u$ is a weak solution of
	\begin{equation}\label{eq:EDP.intro}
		\begin{cases}
				 \mathcal{L}_{s,p}u(x,y)=\lambda |u(x,y)|^{p-2}u(x,y)
				&\text{ in }\Omega,\\	
				u(x,y)=0 &\text{ in }\Omega^c = \mathbb{R}^{n+m}\setminus 
				\Omega.
		\end{cases}
	\end{equation}
	Here $\widetilde{\mathcal{W}}^{s,p}(\Omega)=\{u\in\wwspr\colon
	u\equiv0 \text{ in }\Omega^c\}.$
	We will study the Dirichlet problem for this operator in a companion paper.

	We impose the following assumptions on the data:
	\begin{itemize}
		\item[A1.] $\Omega$ is a
		bounded Lipschitz domain in $\mathbb{R}^{n+m};$
		\item[A2.]  $s\in(0,1),$ and $p\in(1,+\infty).$
	\end{itemize}
	
	Under these conditions we have the following result.
	
	\begin{teo} \label{teo.autov.intro} There exists a
	sequence of eigenvalues $\lambda_n$ such that
	$\lambda_n\to+\infty$ as $n\to+\infty$.
	 Moreover, every eigenfunction
		is in $L^{\infty}(\mathbb{R}^{n+m}).$
		The first eigenvalue (the smallest eigenvalue) is given by
       \begin{equation}\label{eq:autovalor.intro}
			\lambda_1(s,p)\coloneqq\inf
                 \left\{
                     \dfrac{[u]_{\mathcal{W}^{s,p}(
						\mathbb{R}^{n+m})}^p}{\|u\|_{L^p(\Omega)}^p}	
						\colon
                      	u\in\widetilde{\mathcal{W}}^{s,p}(\Omega), 
                      	u\not\equiv0
                 \right\}.
		\end{equation} 
		This eigenvalue $\lambda_1(s,p)$ is simple, isolated and
		 an associated eigenfunction is strictly positive 
		 (or negative) in $\Omega$.
	\end{teo}
	
	Next, we analyze the limit as $s\to 1^-$ of the first eigenvalue 
	obtaining that there is 
	a limit that is the first eigenvalue of a local operator that involve two 
	$p-$Laplacians (one in the $x$ variables and another one in $y$ variables).
	
	\begin{teo} \label{teo.s.to.1}
	Let $\Omega$ is bounded domain in
	$\mathbb{R}^{n+m}$ with smooth boundary, and fix 
	$p\in(1,\infty).$ Then
		\begin{equation}\label{eq:sa1.intro}
		\begin{aligned}
			&\lim_{s\to 1^-}(1-s)\lambda_1(s,p)= \lambda_1(1,p)\\
			& \quad \coloneqq \inf\left\{
			\dfrac{K_{n,p}\|\nabla_x u\|^p_{\lp} 
			+K_{m,p}\|\nabla_y u\|^p_{\lp}}{\displaystyle
			\|u\|_{\lpp}^p}\colon u\in W^{1,p}_0(\Omega), 
			u\not\equiv0\right\},
		\end{aligned}
	\end{equation}
	where the constant $K_{n,p}>0$ depends only on $n$ and $p,$ while 
	$K_{m,p}>0$ depends only on $m$ and $p.$ 
	\end{teo}
	
	Observe that the limit value, $\lambda_1(1,p)$,
	is the first eigenvalue of the following eigenvalue problem
	\[
		\begin{cases}
			-K_{n,p}\Delta_{p,x} u - K_{m,p} \Delta_{p,y}u
			=\lambda 
			|u|^{p-2}u &\text{ in }\Omega,\\
			u=0 &\text{ on }\partial\Omega.
		\end{cases}
	\]

	Concerning the limit as $p\to \infty$ (for a fixed $s$) for the first 
	eigenvalue we have the following result.
	
	\begin{teo} \label{teo.lim.intro} 
	It holds that
	\[
		\lim_{p\to \infty} [\lambda_1(s,p)]^{\nicefrac1p} 
		= \Lambda_{\infty} (s)
	\]
	where
	\[
		\Lambda_{\infty} (s) \coloneqq \inf
		\left\{[u]_{\mathcal{W}^{s,\infty}(\mathbb{R}^{n+m})}\colon 
		u\in \mathcal{W}^{s,\infty}(\mathbb{R}^{n+m}),  
		\|u\|_{L^\infty (\Omega)}=1, u=0 \mbox{ in }\Omega^c\right\}.
	\]
	In addition, the eigenfunctions $u_p$ normalized by 
	$\|u_{p}\|_{L^p (\Omega)}=1$ converge along subsequences 
	$p_n \to \infty$ uniformly to a continuous limit $u_\infty$, 
	that is a nontrivial viscosity solution to  	
	\[
		\begin{cases}
			\max\{A; C\}= \max\{-B;-D; \Lambda_\infty (s) u\} 
			&\mbox{in } \Omega, \\
			u=0 &\mbox{ in }\Omega^c,\\
		\end{cases}
	\]
	with
		\begin{align*}
			&A= \sup_w \frac{u(x,w)-u(x,y)}{|y-w|^{s}},
			&B= \inf_w \frac{u(x,w)- u(x,y)}{|y-w|^{s}},\\
			&C= \sup_z \frac{u (z,y)- u(x,y)}{|x-z|^{s}},
			&D= \inf_z \frac{u(z,y)- u(x,y)}{|x-z|^{s}}.
		\end{align*}
	\end{teo}

	We can give a simple geometric characterization of the limit value 
	$\Lambda_\infty(s)$, 
	this value is related to the maximum distance (measured 
	in a way that involves the exponent $s$, see below) from one point 
	$(x,y) \in \Omega$ to the boundary. In fact,
	$$
		\Lambda_\infty(s) =  \frac{1}{ \displaystyle
		\max_{(x,y) \in \Omega} 
		\displaystyle\min_{(z,w) \in \partial \Omega}  
		(|x-z|^{s}+|y-w|^{s})}.
	$$
	
	That the limit equation is verified in the viscosity sense and 
	involve quotients of the form
	$\frac{u(x,w)-u(x,y)}{|y-w|^{s}}$ is not surprising. In fact, viscosity 
	solutions provide the right framework to deal with limits of 
	$p-$Laplacians as $p\to \infty$, see \cite{ACJ,BBM,JLM},  and quotients
	like the one mentioned above appeared in other related limits, see 
	\cite{Moneau,FP,LL}. 
	What is remarkable in the limit equation is that it involves the limit value 
	$\Lambda_\infty(s)$ and that the
	quotients that appear have perfectly identified the two groups of variables 
	that are present in the fractional pseudo $p-$Laplacian that we 
	introduced here.
	
	Our results say that we can take the limits as $s \to 1^-$
	and as $p\to \infty$ in the first eigenvalue. With the above
	notations we have the following commutative diagram
	\begin{equation}\label{diagrama.conmut}
	    \begin{CD}
			 ( (1-s) \lambda_1 (s,p))^{\nicefrac{1}{p}} @>>
			 {s\to1^-}>(\lambda_1 (1,p))^{\nicefrac{1}{p}}\\
			@V{p\to\infty}VV @VV{p\to\infty}V\\
			\Lambda_\infty (s) @>>{s\to1^-}>\Lambda_{\infty} .
		\end{CD}
	\end{equation}
	Here
	\[
		\Lambda_{\infty} \coloneqq
		 \frac{1}{ \displaystyle
		 \max_{(x,y) \in \Omega} 
		 \displaystyle\min_{(z,w) \in \partial \Omega} (|x-z|+|y-w|)}.
	\]
	The limit 	
	\[
		\lim_{p\to \infty} (\lambda_1 (1,p))^{\nicefrac{1}{p}} 
		= \Lambda_{\infty}
	\]
	can be obtained as in \cite{JLM} using the variational characterization of 
	$\lambda_1 (1,p)$ given in \eqref{eq:sa1.intro}. We omit the details.

	To end this introduction, let us comment on previous results.
	The limit as $p\to\infty$ of the first eigenvalue 
	$\lambda_{p}^D$ of the usual local
	$p$-Laplacian with Dirichlet boundary condition was studied in 
	\cite{JLM,JL}, (see also \cite{BK} for an anisotropic version). 
	In those papers 
	the authors prove that
	$$
		\lambda_{\infty}^D\coloneqq
		\lim_{p\to +\infty}\left(\lambda_{p}^D\right)^{1/p}=
 		\inf \left\{
 		\displaystyle \frac{ \displaystyle 
 		\|\nabla v\|_{L^\infty(\Omega)}}
	    {\displaystyle \|v\|_{L^\infty (\Omega)} }\colon
	    v\in W^{1,\infty}_0 (\Omega), v\not\equiv0\right\}
		= \frac{1}{R},
	$$
	where $R$ is the largest possible radius of a ball contained 
	in $\Omega$.
	In addition, it was shown the existence of extremals, i.e. 
	functions where the above infimum is attained.
	These extremals can be constructed taking the limit as 
	$p\to \infty$ in the eigenfunctions of the $p-$Laplacian 
	eigenvalue problems (see \cite{JLM}) and are viscosity solutions of 
	the following eigenvalue problem 
	(called the infinity eigenvalue problem in the literature)
	\begin{equation*}
		\begin{cases}
			\min \left\{|D u|-\lambda_{\infty}^D u,\, 
			\Delta_{\infty} u \right\}=0 &\text{in }\Omega,\\
			u=0& \mbox{on } \partial \Omega.
		\end{cases}
	\end{equation*}
	The limit operator
	$\Delta_{\infty}$ that appears here
	is the $\infty$-Laplacian given by
	$\Delta_\infty u = -\langle D^2u Du, Du \rangle.$
	Remark that solutions to 
	$\Delta_p v_p =0$ 
	with a Dirichlet data $v_p=f$ on $\partial \Omega$ converge as 
	$p\to \infty$ to the viscosity solution to $\Delta_\infty v =0$ 
	with $v=f$ on $\partial \Omega$, 
	see \cite{ACJ,BBM,CIL}. 
	This operator appears naturally when one considers absolutely minimizing 
	Lipschitz extensions in $\Omega$ of a boundary data $f,$ 
	see \cite{A,ACJ}. Limits of $p-$Laplacians are also relevant in mass
	transfer problems, see \cite{BBP,EG}.

	 On the other hand, the pseudo infinity Laplacian is the second 
	 order nonlinear operator given by
	$\tilde{\Delta}_\infty u = 
	\sum_{i \in I (\nabla u)} u_{x_i x_i} | u_{x_i}|^2,$
	where the sum is taken over the indexes in $I (\nabla u) = \{ i\,:
	\, | u_{x_i}| =\max_j | u_{x_j}|\}$. This operator, as happens for the 
	usual infinity Laplacian, also appears naturally as a limit of $p-$Laplace
	type problems. In fact, any possible limit of $u_p$, solutions to
	$\tilde{\Delta}_p u =\sum_{i=1}^N ( | u_{x_i}|^{p-2}  u_{x_i} )_{x_i} =0,$
	is a viscosity solution to $\tilde{\Delta}_\infty u = 0$.
	A proof of this fact is contained in \cite{BK}, where are also studied the 
	eigenvalue problem for this operator.

	Concerning regularity, we mention 
	\cite{S} where it it proved that infinity harmonic functions, that is, 	
	viscosity solutions to $ -\Delta_\infty u =0$, are $C^1$  in two dimensions 
	and \cite{Evans-Smart, Evans-Smart2} where it is proved differentiability in 
	any dimension. For the pseudo infinity Laplacian, we refer here to solutions 
	to $\tilde{\Delta}_\infty u = 0$, the optimal regularity is Lipschitz 
	continuity, see \cite{RS}.

	For references concerning nonlocal fractional
	problems we refer to \cite{DRV,Jilka,LL,Bisci1,Bisci2, Bisci3,Hich} and 
	references therein. For limits as $p\to +\infty$ in nonlocal 
	$p-$Laplacian problems and its relation with optimal mass transport 
	we refer to \cite{Jilka} (eigenvalue 
	problems were not considered there). 
	
	Finally, concerning limits as $p\to \infty$ in fractional eigenvalue 	
	problems, we mention \cite{Brasco,FP,JL}.
	In \cite{JL} the limit of the first eigenvalue for the fractional 
	$p-$Laplacian is studied while in \cite{FP} higher eigenvalues are 
	considered. We borrow ideas and techniques from these papers. In particular, 
	when we prove the fact that there is a limit 
	problem that is verified in the viscosity sense. For example, the fact that 
	continuous weak solutions to our pseudo fractional $p-$Laplacian are 
	viscosity solutions runs exactly as in \cite{JL} and hence we omit the 
	details here.

	The paper is organized as follows: In Section \ref{Prel} we collect some 
	preliminary results; in Section \ref{autovalor} we deal with our eigenvalue 
	problem and prove Theorem \ref{teo.autov.intro}; in Section \ref{limites} we analyze the limit as $s\to 1^-$, Theorem \ref{teo.s.to.1}; finally, in Section 
	\ref{limitep} we study the limit as $p\to \infty$ proving 
	Theorem \ref{teo.lim.intro}.
	
\section{Preliminaries}\label{Prel}
\setcounter{equation}{0}
	Throughout this section $s\in(0,1),$  $p\in (1,+\infty],$
	$\Omega$ is an open set of $\mathbb{R}^{n+m}.$ 
	We henceforth use the notation:
	\begin{itemize}
		\item  $(x,y)=(x_1,\dots,x_n,x_{n+1},\dots,x_{n+m})
				\in \mathbb{R}^{n+m}$ with $x=(x_1,\dots,x_n)\in \mathbb{R}^{n}$ and
				$y=(x_{n+1},\dots,x_{n+m}) \in \mathbb{R}^{m};$
		\item $\Omega^2=\Omega\times \Omega;$
		\item $\Omega_x=\{y\in \mathbb{R}^m\colon (x,y)\in \Omega\},$
		and 
		$\Omega_y=\{x\in \mathbb{R}^n\colon (x,y)\in \Omega\};$
		\item $B^N(x,r)$ denotes the ball of $N-$ball of radius 
			$r$ and center $x,$ and 
			$\omega_N$ denotes the $(N-1)-$dimensional Hausdorff measure
			of the $N-$sphere of radius 1;
		\item $(a)^{p-1}=|a|^{p-2}a.$
	\end{itemize}

	Given a measurable function 
	$u\colon\Omega\to\mathbb{R},$ we set for $p<+\infty,$
	\[
		\|u\|_{\scriptstyle \lpp}^p
		\coloneqq\int_{\Omega}|u(x,y)|^p\,dxdy,
	\]
	\[
		|u|_{\scriptstyle\wspp}^p
				=\int_{\Omega^2}
		\dfrac{|u(x,y)-u(z,w)|^p}{|(x,y)-(z,w)|^{n+m+sp}} dxdydzdw, 	
	\]
	\begin{align*}
		[u]_{\scriptstyle\wspp}^p
				=\int_{\Omega}\int_{\Omega_y}&
		\dfrac{|u(x,y)-u(z,y)|^p}{|x-z|^{n+sp}} dzdxdy\\
		&+
		\int_{\Omega}\int_{\Omega_x}
		\dfrac{|u(x,y)-u(x,w)|^p}{|y-w|^{m+sp}} dwdxdy 	
	\end{align*}
	 and for $p=+\infty,$ 
	\[
		|u|_{\scriptstyle W^{\scriptscriptstyle s,\infty}(\Omega)}
		=\sup\left\{
		\dfrac{|u(x,y)-u(z,y)|}{|(x,y)-(z,w)|^{s}}\colon
		(x,y)\neq(z,w)\in \Omega\right\}
		{=|u|_{C^{0,s}(\Omega)}},
	\]
	\begin{align*}
		[u]_{\scriptstyle \mathcal{W}^{\scriptscriptstyle s,\infty}
		(\Omega)}
		&=\max\left\{\sup\left\{
		\dfrac{|u(x,y)-u(z,y)|}{|x-z|^{s}}\colon (x,y)\neq(z,y)\in \Omega
		\right\}\right.;\\
		&\qquad\qquad\left.\sup\left\{
		\dfrac{|u(x,y)-u(x,w)|}{|y-w|^{s}}\colon (x,y)\neq(x,w)
		\in \Omega\right\}
		\right\}.
	\end{align*}
	
	We denote by $\wsp$ (here $p$ can be $+\infty$) the usual fractional Sobolev space, that is
	$
		\wsp\coloneqq \left \{ u\in \lp
		\colon |u|_{\scriptstyle \wspp} <+\infty \right \}.
	$
	
	We introduce
	the space $\displaystyle\wwsp$ 
	(again here $p$ can be $+\infty$) as follows:
	\[
		\wwsp\coloneqq
		\left\{u\in \lp\colon
		[u]_{\scriptstyle\wwspp}^p<\infty
		\right\}.
	\]
This space is a Banach space. We state this as a proposition but we omit its proof that is
standard.
	
	\begin{prop}
		The space $\displaystyle\wwsp$ 
		endowed with the norm
		\[
			\| u\|_{\scriptstyle\wwspp}=
			\left(\| u\|_{\scriptstyle \lpp}^p
			+[u]_{\scriptstyle \wwspp}^p \right)^{\nicefrac1p}
		\]
		is a Banach space. Moreover $\displaystyle\wwsp$ is separable
		for $1\le p<+\infty$ and it is reflexive for $1<p<\infty.$
	\end{prop}
	
	For  $u\colon \Omega\to\mathbb{R}$ a measurable function,
	we set
	\[
		u_{+}(x,y)=\max\{u(x,y),0\}\quad\mbox{ and }\quad
		u_{-}(x,y)=\min\{-u(x,y),0\}.
	\]
	Observe that
	\[
			|u_{\pm}(x,y)-u_{\pm}(z,w)|\le |u(x,y)-u(z,w)|
	\]
	for all $(x,y),(z,w)\in\Omega.$ Therefore, we have
	
	\begin{lem}\label{lema:aux1}
		Let $\mathcal{X}=\wsp$ or
		$\wwsp.$ If
		$u\in \mathcal{X}$ then $u_{+}, u_{-}\in
		\mathcal{X}.$
	\end{lem}

	For $1\le p<\infty,$  we denote by 
	$\widetilde{\mathcal{W}}^{s,p}(\Omega)$ the space of all
	$u\in \wwsp$ such that $\tilde{u}\in \wwspr$ where $\tilde{u}$ is
	the extension by zero of $u.$

	The next result can be found in \cite{AF,DD}.
						
	\begin{teo}\label{teo:compacemb}
			Under the assumptions {\rm A1} and {\rm A2}
			we have that
			\begin{itemize}
				\item If $sp<n+m,$ then  
				$\wsp$
				 is compactly embedded  in $L^q(\Omega)$
				 for all $1\le q<p^\star_s=\nicefrac{(n+m)p}{(n+m-sp)}.$
				\item If $sp=n+m,$ then  $\wsp$
				 is compactly embedded  in $L^q(\Omega)$ for all 
				 $1\le q<\infty.$
				\item If $sp>n+m,$ then  $\wsp$
				is compactly embedded  in
				$C^{0,\lambda}(\overline{\Omega})$ with
				$\lambda<s-\nicefrac{(n+m)}{p}.$
			\end{itemize}
	\end{teo}

	\begin{lem}\label{lema:contincl}
	 	Let $\Omega_1$ and $\Omega_2$ be open subsets of 
	 	$\mathbb{R}^n$ and $\mathbb{R}^m$ respectively. 
	 	If $\Omega=\Omega_1\times\Omega_2,$
	 	and $p\in[1,+\infty),$ then
	 	$\wwsp$ is continuously embedded in
	 	$\wsp.$ Moreover, there exists a constant $C=C(n,m)$
	 	such that 
	 	\[
	 		|u|_{\scriptstyle\wspp}^p\le C[u]_{\scriptstyle\wwspp}
	 	\]
	 	for all $u\in\wwsp.$
	 \end{lem}
	 \begin{proof}
			Let $u\in\wwsp.$ We have
			\begin{equation}\label{eq:continc1}
				\begin{aligned}
				|u|_{\wsp}^p  &=
				\int_{\Omega^2}
	 			\dfrac{|u(x,y)-u(z,w)|^p}{|(x,y)-(z,w)|^{n+m+sp}}
	 			\, dxdydzdw\\
	 			&\le 2^{p-1}
	 			\int_{\Omega^2}
	 			\dfrac{|u(x,y)-u(z,y)|^p}{|(x,y)-(z,w)|^{n+m+sp}}
	 			dxdydzdw\\
	 			&\qquad +2^{p-1}
	 			\int_{\Omega^2}
	 			\dfrac{|u(z,y)-u(z,w)|^p}{|(x,y)-(z,w)|^{n+m+sp}}
	 			dxdydzdw\\
	 			&=2^{p-1}I_1+2^{p-1}I_2.
				\end{aligned}
			\end{equation}
			Now, we observe that
			\begin{align*}
				I_1&=
				\int_{\Omega^2}
	 			\dfrac{|u(x,y)-u(z,y)|^p}{|(x,y)-(z,w)|^{n+m+sp}}
	 			dxdydzdw\\
	 			&\le\int_{\Omega}\int_{\Omega_2}\int_{\mathbb{R}^m}
	 			\dfrac{|u(x,y)-u(z,y)|^p}{|(x,y)-(z,w)|^{n+m+sp}}
	 			dwdzdxdy\\
	 			&\le\int_{\Omega}\int_{\Omega_2}
	 			\dfrac{|u(x,y)-u(z,y)|^p}{|x-z|^{n+sp}}\int_{\mathbb{R}^m}
	 			\dfrac{|x-z|^{n+sp}dw }
	 			{\left(|x-z|^2+|y-w|^2\right)^{\frac{n+m+sp}2}}
	 			dzdxdy\\
	 			&=\omega_m\int_{\Omega}\int_{\Omega_2}
	 			\dfrac{|u(x,y)-u(z,y)|^p}{|x-z|^{n+sp}}
	 			dzdxdy
	 			\int_{0}^{+\infty}
	 			\dfrac{r^{m-1}}{\left(1+r^2\right)^{\frac{n+m+sp}2}}dr.
			\end{align*}
			Since
			\begin{align*}
				\int_{0}^{+\infty}
	 			\dfrac{r^{m-1}}{\left(1+r^2\right)^{\frac{n+m+sp}2}}dr
	 			&\le\int_{0}^{1}
	 			r^{m-1}dr
	 			+\int_{1}^{+\infty}
	 			\dfrac{1}{r^{n+sp+1}}dr=\dfrac1m+\dfrac{1}{n+sp}
			\end{align*}
			we have that
			\begin{equation}\label{eq:continc2}
				I_1\le 2\omega_m
				\int_{\Omega}\int_{\Omega_2}
	 			\dfrac{|u(x,y)-u(z,y)|^p}{|x-z|^{n+sp}}
	 			dzdxdy.
			\end{equation}
			One can also, in an analogous way, obtain
		 	\begin{equation}\label{eq:continc3}
				I_2\le 2\omega_n
				\int_{\Omega}\int_{\Omega_1}
	 			\dfrac{|u(x,y)-u(x,w)|^p}{|y-w|^{m+sp}}
	 			dwdxdy.
			\end{equation}
			
			By \eqref{eq:continc1}, \eqref{eq:continc2} and
			\eqref{eq:continc3}, we get
			\[
				|u|_{\wsp}\le
				C(n,m)[u]_{\wwsp}.
			\]
			This completes the proof.
		\end{proof}

		\begin{remark}
			If $p=\infty,$ it is straightforward to show that 
			$W^{s,\infty}(\Omega)\subset\mathcal{W}^{s,\infty}(\Omega).$ 
			Moreover, if 
			$\Omega=\Omega_1\times \Omega_2$ then	
			$\mathcal{W}^{s,\infty}(\Omega)=
			W^{s,\infty}(\Omega).$		
		\end{remark}	
	
	\begin{lem}\label{lema:inclu2}
	 	Let $\Omega$ be an open subset of 
	 	$\mathbb{R}^{n+m}$ and $p\in(1,\infty).$ 
	 	 If $0<t<s<1$ then
	 			$\wwsp\subset\mathcal{W}^{t,p}(\Omega),$
	 			and the embedding is continuous.
	 			Moreover 
	 			\begin{equation}\label{eq:inclu1}
					[u]_{\scriptstyle 
					\mathcal{W}^{\scriptscriptstyle t,p}(\Omega)}^p
					\le 
					[u]_{\mathcal{W}^{\scriptscriptstyle s,p}(\Omega)}^p
					+ \dfrac{2^p(\omega_{n}+\omega_{m})}{tp}
					\|u\|_{\scriptstyle
					\lpp}^p \qquad\forall u\in\wwsp.
				\end{equation}	
	 	\begin{proof}
			Let $u\in\wwsp.$ Observe that,
			\begin{align*}
				\int_{\Omega}{\int_{\Omega_y}}
				\dfrac{|u(x,y)-u(z,y)|^p}{|x-z|^{n+tp}}dzdxdy\le
				\int_{\Omega}{\int_{A_y}}
				&\dfrac{|u(x,y)-u(z,y)|^p}{|x-z|^{n+tp}}dzdxdy\\
				&+\int_{\Omega}{\int_{A_y^c}}
				\dfrac{|u(x,y)-u(z,y)|^p}{|x-z|^{n+tp}}dzdxdy
			\end{align*}
			where $A_y=\{z\in\Omega_y\colon |z-x|<1\}.$ Since $t<s,$
			we have that
			\begin{align*}
				&\int_{\Omega}{\int_{\Omega_y}}
				\dfrac{|u(x,y)-u(z,y)|^p}{|x-z|^{n+tp}}dzdxdy\le\\
				&\le\int_{\Omega}{\int_{A_y}}
				\dfrac{|u(x,y)-u(z,y)|^p}{|x-z|^{n+sp}}dzdxdy
				+2^{p-1}\int_{\Omega}{\int_{A_y^c}}
				\dfrac{|u(x,y)|^p+|u(z,y)|^p}{|x-z|^{n+tp}}dzdxdy\\
				&\le\int_{\Omega}{\int_{A_y}}
				\dfrac{|u(x,y)-u(z,y)|^p}{|x-z|^{n+sp}}dzdxdy
				+2^p\int_{\Omega}{\int_{A_y^c}}
				\dfrac{|u(x,y)|^p}{|x-z|^{n+tp}}dzdxdy\\
				&\le\int_{\Omega}{\int_{A_y}}
				\dfrac{|u(x,y)-u(z,y)|^p}{|x-z|^{n+sp}}dzdxdy
				+\dfrac{2^p\omega_{n}}{tp}\int_{\Omega}
				|u(x,y)|^pdxdy.
			\end{align*}
			
			Similarly,
			\begin{align*}
				\int_{\Omega}{\int_{\Omega_x}}
				&\dfrac{|u(x,y)-u(x,w)|^p}{|y-w|^{n+tp}}dzdxdy\le\\
				&\le\int_{\Omega}{\int_{A_x}}
				\dfrac{|u(x,y)-u(z,y)|^p}{|x-z|^{n+sp}}dzdxdy
				+\dfrac{2^p\omega_{m}}{tp}\int_{\Omega}
				|u(x,y)|^pdxdy ,
			\end{align*}
			where $A_x=\{w\in\Omega_x\colon |y-w|<1\}.$
			Therefore \eqref{eq:inclu1} holds.
		\end{proof}
	 \end{lem}
	 
	 Finally, we prove a Poincar\'e type inequality.
	 
	 \begin{lem}\label{lema:Poicare} 
	 	Let $\Omega$ be an open bounded subset of 
	 	$\mathbb{R}^{n+m},$ $s\in(0,1)$ and $p\in(1,\infty).$
	 	Then there is a positive constant $C$ such that
	 	\[
	 		\|u\|_{\lp}\le C 
	 		[u]_{\mathcal{W}^{\scriptscriptstyle s,p}
	 		(\mathbb{R}^{n+m})}
	 		\quad\forall u\in
	 		\widetilde{\mathcal{W}}^{\scriptscriptstyle s,p}
	 		(\Omega).
	 	\] 
	\end{lem}
	\begin{proof}
		Let 
		$u\in\widetilde{\mathcal{W}}^{\scriptscriptstyle s,p}
		(\Omega)$ and $d=2\diam(\Omega)$. It holds that
		\[
				[u]_{\mathcal{W}^{\scriptscriptstyle s,p}
	 				(\mathbb{R}^{n+m})}^p\ge
	 				\int_{\Omega}|u(x,y)|^p\int_{\mathbb{R}^{n+m}\setminus 	
	 				B^n(x,d)}	 
	 				\dfrac{dz}{|x-z|^{n+sp}}\ge
	 				\dfrac{\omega_n d^{-sp}}{sp}
	 				\|u\|_{\lp}^p.				
		\]
	\end{proof}	 	 
	
\section{The first eigenvalue}\label{autovalor}
\setcounter{equation}{0}

	Under assumptions \rm{A1} and \rm{A2}, a natural definition of an
	eigenvalue is a real value $\lambda$ for which there exists
	$u\in \widetilde{\mathcal{W}}^{s,p}(\Omega)\setminus\{0\}$ such
	that $u$ is a weak solution of
	\begin{equation}\label{eq:EDP}
		\begin{cases}
				\mathcal{L}_{s,p}u(x,y)=\lambda (u(x,y))^{p-1}
				&\text{ in }\Omega,\\	
				u(x,y)=0 &\text{ in }\Omega^c,
		\end{cases}
	\end{equation}
	that is
	\[
		\mathcal{H}_{s,p}(u,v)=\lambda
		\int_{\Omega} (u(x,y))^{p-1}
		v(x,y)\, dxdy\qquad \forall v\in
		\widetilde{\mathcal{W}}^{s,p}(\Omega).
	\]
	The function $u$ is called a corresponding eigenfunction.
	Here
	\begin{align*}
		\mathcal{H}_{s,p}(u,v)\coloneqq&
		\int_{\mathbb{R}^{n+m}}\int_{\mathbb{R}^{n}}
		\dfrac{(u(x,y)-u(z,y))^{p-1}
		(v(x,y)-v(z,y))}{|x-z|^{n+sp}}
		dzdxdy\\
		&+\int_{\mathbb{R}^{n+m}}\int_{\mathbb{R}^m}
		\dfrac{(u(x,y)-u(x,w))^{p-1}
		(v(x,y)-v(x,w))}{|y-w|^{m+sp}}dwdxdy.
	\end{align*}
	
	Observe that
	\begin{equation}\label{eq:DP1}
		\mathcal{H}_{s,p}(u,u)= [u]_{\mathcal{W}^{s,p}(
		\mathbb{R}^{n+m})}^p\qquad
		 \forall u\in \mathcal{W}^{s,p}(\mathbb{R}^{n+m}),
	\end{equation}
	and, by H\"older's inequality,
	\begin{equation}\label{eq:DP2}
		\mathcal{H}_{s,p}(u,v)\le 2[u]_{\mathcal{W}^{s,p}(
		\mathbb{R}^{n+m})}^{p-1}
		[v]_{\mathcal{W}^{s,p}(
		\mathbb{R}^{n+m})}
		\qquad \forall u,v\in
		\mathcal{W}^{s,p}(\mathbb{R}^{n+m}).
	\end{equation}

	Observe that, when $\lambda$ is an eigenvalue, then there is
    $u\in \widetilde{\mathcal{W}}^{s,p}(\Omega)\setminus\{0\}$
    such that
    \[
		\mathcal{H}_{s,p}(u,u)=\lambda \int_{\Omega}|u(x,y)|^{p}dxdy.
	\]
    Then, we have
    that
    \begin{equation}\label{eq:autcarac}
		\lambda =\dfrac{[u]_{\mathcal{W}^{s,p}(
		\mathbb{R}^{n+m})}^p}{\|u\|_{L^p(\Omega)}^p}\ge 0.
    \end{equation}
	
	By a standard compactness argument,
	we have the following result.
	
	\begin{teo}\label{teo:1erAut}
		Under the assumptions {\rm{A1}} and
		{\rm{A2}}, the first eigenvalue is given by
       \begin{equation}\label{eq:autovalor}
			\lambda_1(s,p)\coloneqq\inf
                 \left\{
                     \dfrac{[u]_{\mathcal{W}^{s,p}(
						\mathbb{R}^{n+m})}^p}{\|u\|_{\lpp}^p}	
						\colon
                      	u\in\widetilde{\mathcal{W}}^{s,p}(\Omega), 
                      	u\not\equiv0
                 \right\}.
		\end{equation}
	\end{teo}
	
	\begin{proof} 
		Consider a minimizing sequence $u_n$ normalized according to 
		$\|u_n\|_{L^p(\Omega)}=1$. Then, as $u_n$ in bounded in 
		$\widetilde{\mathcal{W}}^{s,p}(\Omega),$ by Lemma \ref{lema:contincl}
		and Theorem \ref{teo:compacemb}, there is a subsequence 
		such that 
		$u_{n_j} \rightharpoonup u$ weakly 
		in $\widetilde{\mathcal{W}}^{s,p}(\Omega)$ 
		and $u_{n_j} \to u$ strongly in $L^p(\Omega)$. 
		Therefore, $u$ is a nontrivial minimizer to 
		the variational problem defining $\lambda_1(s,p)$. The fact that this
		minimizer is a weak solution to \eqref{eq:EDP} is straightforward and can 
		be obtained from the arguments in \cite{LL}.

		To finish the proof we just observe that any other eigenfunction associated 
		with an eigenvalue $\lambda$ verifies 
		$$
			\lambda =
			\dfrac{[u]_{\mathcal{W}^{s,p}(
			\mathbb{R}^{n+m})}^p}{\|u\|_{L^p(\Omega)}^p}	\geq \lambda_1(s,p),
		$$
		and then we get that $\lambda_1(s,p)$ is the first eigenvalue.
		\end{proof}

		Now we observe that using a topological tool (the genus) we can construct
		an unbounded sequence of eigenvalues.

		\begin{teo} \label{teo2}
			Assume {\rm{A1}} and {\rm{A2}}. 
			There is a sequence of eigenvalues $\lambda_n$ such that
			$\lambda_n\to+\infty$ as $n\to+\infty$.
		\end{teo}

		\begin{proof} 
			We follow ideas from \cite{GAP} and hence we omit the details.
			Let us consider 
			\[
				M_\alpha = \{u \in \widetilde{\mathcal{W}}^{s,p}(
				\Omega)\colon  [u]_{\mathcal{W}^{s,p}(
				\mathbb{R}^{n+m}) }= p \alpha \}
			\] 
			and
			\[ 
				\varphi (u) = \frac{1}{p}
					\int_{\Omega} |u(x,y)|^p\, dxdy.
			\]
			We are looking for critical points
			of $\varphi$ restricted to the manifold $M_\alpha$ using a minimax
			technique.
			We consider the class
			\[
  				\Sigma = \{A\subset \widetilde{\mathcal{W}}^{s,p}(
				\Omega)\setminus\{0\}
				\colon A \mbox{ is closed, } A=-A\}.
  			\]
			Over this class we define the genus, 
			$\gamma\colon\Sigma\to {\mathbb{N}}\cup\{\infty\}$, as 
			\[
				\gamma(A) = \min\{k\in {\mathbb{N}}\colon
				\mbox{there exists } \phi\in C(A,{{\mathbb{R}}}^k-\{0\}),
				\ \phi(x)=-\phi(-x)\}.
			\]
			Now, we let $C_k = \{ C \subset M_\alpha \colon C 
			\mbox{ is compact, symmetric and } \gamma ( C) \le k \} $ 
			and let
			\begin{equation}
					\label{betak} \beta_k 
					= \sup_{C \in C_k} \min_{u \in C} \varphi(u). 
			\end{equation}
			Then $\beta_k >0$ and there exists $u_k \in
			M_\alpha$ such that $\varphi (u_k) = \beta_k$ and $ u_k$ is a weak
			eigenfuction with $\lambda_k = \alpha / \beta_k $.
		\end{proof}

	The following lemma shows that the eigenfunctions are bounded.
	
	\begin{lem}\label{lema:cotainfty}
		Under assumptions {\rm{A1}} and {\rm{A2}},  if $u$ is an eigenfunction
		associated to some eigenvalue $\lambda$ then
		$u\in L^{\infty}(\mathbb{R}^{n+m}).$
	\end{lem}
	\begin{proof}
		In this proof we follow ideas form \cite{FP}.
		
		If $ps>n+m,$ by Lemma \ref{lema:contincl} and
		Theorem \ref{teo:compacemb}, then the assertion holds.
		From now on, we suppose that $sp\le n+m.$
		
		We will show that if $\|u_+\|_{L^{p}(\Omega)}\le \delta$
		then $u_+$ is bounded, where $\delta>0$ is some small constant to 
		be determined.
		Let $k\in\mathbb{N}_0,$ we define the function $u_k$ by
		\[
			u_k(x,y)\coloneqq(u(x,y)-1+2^{-k})_+.
		\]
		Observe that, $u_0= u_+$ and for any $k\in\mathbb{N}_0$ we have that
		$u_k\in \widetilde{W}^{s,p}(\Omega)$ verifies
		\begin{equation}\label{eq:lci1}
		\begin{aligned}
			&u_{k+1}\le u_k \text{ a.e. } \mathbb{R}^{n+m},\\
			&u <(2^{k+1}-1)u_k\text{ in } \{u_{k+1}>0\},\\
			&\{u_{k+1}>0\}\subset\{u_k>2^{-(k+1)}\}.
		\end{aligned}
		\end{equation}
		
		Now, for any function $v\colon\mathbb{R}^{n+m}
		\to \mathbb{R}$, it holds that
		\[
			|v_+(x,y)-v_+(z,w)|^p\le |v(x,y)-v(z,w)|^{p-1}(v_+(x,y)-v_+(x,w))
		\]
		for all $(x,y),(z,w)\in\mathbb{R}^{n+m}.$
		Then
		\[
			[u_{k+1}]_{\mathcal{W}^{s,p}(\mathbb{R}^{n+m})}^p\le
			\mathcal{H}_{s,p}(u,u_{k+1})=\lambda\int_{\Omega}
			(u(x,y))^{p-1} u_{k+1}(x,y) \, dxdy
		\]
		for all $k\in\mathbb{N}_0.$
		Hence, by \eqref{eq:lci1} and H\"older's inequality, we get
		\begin{equation}\label{eq:lci2}
		\begin{aligned}
			[u_{k+1}]_{\mathcal{W}^{s,p}(\mathbb{R}^{n+m})}^p
			&\le \lambda\int_{\Omega} (u(x,y))^{p-1} u_{k+1}(x,y) \, dxdy\\
			&\le (2^{k+1}-1)^{p-1}\lambda\|u_k\|_{L^p(\Omega)}^p\\
		\end{aligned}
		\end{equation}
		for all $k\in\mathbb{N}_0.$

		On the other hand, in the case $sp<n+m,$ using H\"older's inequality,
		 Lemma \ref{lema:contincl} and
		Theorem \ref{teo:compacemb}, the formulas in \eqref{eq:lci1}, and
		Chebyshev's inequality, for any $k\in\mathbb{N}_0$ we have that
		\begin{equation}\label{eq:lci3}
			\begin{aligned}
				\|u_{k+1}\|_{L^{p}(\Omega)}^p&\le
				\|u_{k+1}\|_{L^{p_s^\star}(\Omega)}^p
				|\{u_{k+1}>0\}|^{\nicefrac{sp}{(n+m)}}\\
				&\le C[u_{k+1}]_{\mathcal{W}^{s,p}(\mathbb{R}^{n+m})}^p
				|\{u_{k}>2^{-(k+1)}\}|^{\nicefrac{sp}{(n+m)}}\\
				&\le C[u_{k+1}]_{\mathcal{W}^{s,p}(\mathbb{R}^{n+m})}^p
				\left(2^{(k+1)p}\|u_k\|^p_{L^{p}(\Omega)}
				\right)^{\nicefrac{sp}{(n+m)}},
			\end{aligned}
		\end{equation}
		where $C$ is a constant independent of $k.$
		Then, by \eqref{eq:lci2} and \eqref{eq:lci3}, for any
		$k\in\mathbb{N}_0$ we obtain
		\begin{equation}\label{eq:lci4}
		\begin{aligned}
				\|u_{k+1}\|_{L^{p}(\Omega)}^p
				&\le C
				\left(2^{(k+1)p}\|u_k\|^p_{L^{p}(\Omega)}
				\right)^{1+\alpha},
		\end{aligned}
		\end{equation}
		where $C$ is a constant independent of $k$ and
		$\alpha=\nicefrac{sp}{(n+m)}>0.$

		Arguing similarly, in the case $sp=n+m,$ taking $r>p$ and
		proceeding as in the previous case, $sp<n+m$ (with $r$ in place of
		$p_s^\star$), we obtain that \eqref{eq:lci4} holds with
		$\alpha=1-\nicefrac{p}r>0.$

		Therefore, if $sp\le n+m$, there exist $\alpha>0$ and
		a constant $C>1$  such that
		\begin{equation}\label{eq:lci5}
				\|u_{k+1}\|_{L^{p}(\Omega)}^p
				\le C^k \left(\|u_k\|^p_{L^{p}(\Omega)}
				\right)^{1+\alpha},
		\end{equation}
		for any $k\in\mathbb{N}_0.$
		Hence, if $\|u_0\|_{L^{p}(\Omega)}^p
		=\|u_+\|_{L^{p}(\Omega)}^p\le C^{\nicefrac{-1}{\alpha^2}}
		\eqqcolon\delta^p$ then $u_{k}\to0$ strongly in
		$L^{p}(\Omega).$
		But $u_k\to(u-1)_+$ a.e in $\mathbb{R}^{n+m},$ then we conclude that
		$(u-1)_+\equiv0$ in $\mathbb{R}^{n+m}.$
		Therefore, $u_+$ is bounded.
		
		Taking $-u$ in place of $u$ we have
		that $u_-$ is bounded if
		$\|u_-\|_{L^{p}(\Omega)}<\delta.$ 
		
		Hence, as we can multiply an eigenfunction $u$ by a small constant in order to obtain
		$\|u_+\|_{L^{p}(\Omega)}$ and $\|u_-\|_{L^{p}(\Omega)} <\delta$, 
		we conclude that
		$u$ is bounded.
	\end{proof}

   Our next goal is to show that if $u$ is a eigenfunction associated with
   $\lambda_{1}(s,p)$ then $u$ does not change sign. Before showing this
   result we need the following two technical lemmas.

   \begin{lem}\label{lema:DKP}
   		Assume {\rm{A1}} and {\rm{A2}}.
   		If $u\in \widetilde{\mathcal{W}}^{s,p}
   		(\Omega)$ is such that
		\begin{equation}\label{eq:supersol}
			\mathcal{H}_{s,p}(u,v)\ge0\quad\forall v\in
   			\widetilde{\mathcal{W}}^{s,p}(\Omega), v\ge0 \text{ in } \Omega.
		\end{equation}
   		and
   		$u\ge0$ in $B^n(x_0,R)\times B^m(y_0,R)\subset\subset \Omega$
   		for some $R>0$ then for any $d>0$ and $0<2r<R$ there holds
   		\begin{equation}\label{eq:DKP}
			\begin{aligned}
				& \int_{B_r^m}\int_{B_r^n}\int_{B_r^n}
   				\dfrac{1}{|x-z|^{n+sp}}\left|\log
   				\left(\dfrac{u(x,y)+d}{u(z,y)+d}
   				\right)\right|^pdzdxdy\\&
   				+\int_{B_r^n}\int_{B_r^m}\int_{B_r^n}
   				\dfrac{1}{|y-w|^{m+sp}}\left|\log
   				\left(\dfrac{u(x,y)+d}{u(x,w)+d}
   				\right)\right|^pdwdxdy\\
   				& \qquad \qquad \le Cr^{n+m-sp} \Biggl\{
   				\dfrac{r^{sp}}{d^{p-1}r^m}
   				\int_{\mathbb{R}^m}\int_{(B^n_R)^c}
   				\dfrac{u_{-}(x,y)^p}{|x-x_0|^{n+sp}}dxdy\\& \qquad\qquad\qquad
				\qquad 
				+
   				\dfrac{r^{sp}}{d^{p-1}r^n}
   				\int_{\mathbb{R}^n}\int_{(B^m_R)^c}
   				\dfrac{u_{-}(x,y)^p}{|y-y_0|^{m+sp}}dydx +1
   				\Biggl\}
			\end{aligned}
		\end{equation}
		where $B^n_{\rho}=B^n(x_0,\rho),$  $B^m_{\rho}=B^m(y_0,\rho)$ and $C=C(n,m,p,s)>0$
		is a constant.
	\end{lem}
	\begin{proof}
		Let $d>0,$ $r\in(0,\nicefrac{R}2)$,
		\begin{align*}
			&\phi\in C_0^{\infty}(B^n_{\nicefrac{3r}2}), \quad
			0\le \phi\le 1, \quad \phi\equiv1 \text{ in } B_r^n,\quad
			|D_x\phi|<\frac{c}r \text{ in } B_{\nicefrac{3r}2}^n,
			\text{ and }\\
			&\psi\in C_0^{\infty}(B^m_{\nicefrac{3r}2}), \quad
			0\le \psi\le 1, \quad \psi\equiv1 \text{ in } B_r^m,\quad
			|D_x\psi|<\frac{c}r \text{ in } B_{\nicefrac{3r}2}^m.
		\end{align*}
		
		Taking
		$v(x,y)=\phi^p(x)\psi^p(y)(u(x,y)+d)^{1-p}$
		as test function in \eqref{eq:supersol} and following
		the proof of Lemma 1.3 in \cite{DKP}, we get \eqref{eq:DKP}.
	\end{proof}
	
	\begin{lem}\label{lema:positivo}
		Assume {\rm{A1}} and {\rm{A2}}. If  $\Omega$ is 
		connected and $u\in \widetilde{\mathcal{W}}^{s,p}
   		(\Omega)$ is such that
   		\[
   			\mathcal{H}_{s,p}(u,v)\ge0\quad\forall v\in
   			\widetilde{\mathcal{W}}^{s,p}(\Omega), v\ge0 \text{ in } \Omega,
   		\]
   		$u\ge0$ in $\Omega$ and  $u\not\equiv0$ in $\Omega$ then
   		$u>0$ in $\Omega.$
	\end{lem}
	\begin{proof}
		In this proof we borrow ideas from \cite{BF}. Since $\Omega$ is
		a bounded connected open set, it is enough to prove that $u>0$ in
		$K$ for any $K\subset\subset \Omega$ a connected 
		compact set such that $u\not\equiv0$ in $K.$
		
		Let $K\subset\subset \Omega$ be a connected compact set such that
		$u\neq0$ in $K$. Then there exists
		$r>0$ such that
		\[
			K\subset \left \{ (x,y)\in \Omega\colon
			\max_{(z,w)\in\partial \Omega}\{|z-x|,|w-y|\}>2r \right \}.
		\]
		Since $K$ is compact, there exists $\{(x_j,y_j)\}_{j=1}^k\subset K$
		such that
		\begin{equation}\label{eq:inc_int}
			K\subset\bigcup_{j=1}^k B^n_j\times B^m_j,
			\quad \text{	and }  \quad
			|(B^n_j \times B^m_j)\cap (B^n_{j+1} \times B^m_{j+1})|>0
		\end{equation}
		for any $j\in\{1,\dots,k-1\},$ where
		$B_j^n=B^n(x_j,\nicefrac{r}2)$ and
		$B_j^m=B^m(y_j,\nicefrac{r}2).$

		To obtain a contradiction, suppose that
		$|\{(x,y)\colon u(x,y)=0\}\cap K|>0$ then there exists
		$j\in\{1,\dots,k\}$ such that
		\[
			Z=\{(x,y)\colon u(x,y)=0\}\cap(B^n_j\times
			 B^m_j)
		\]
		has positive measure.
		
		Given $d>0,$ we define
		$$
			F_d\colon B^n_j\times B^m_j\to \mathbb{R}\quad \mbox{ by } \quad
			F_d(x,y)=\log\left(1+\dfrac{u(x,y)}{d}\right).
		$$
		Then, for any $(x,y)\in B^n(x_j,\nicefrac{r}2)\times
			 B^m(y_j,\nicefrac{r}2)$ and $(z,w)\in Z$ we have
		\begin{align*}
			F_d(z,w)&=0\\
			|F_d(x,y)|^p&=|F(x,y)-F(z,w)|^p\\
			 &\le 2^{p-1}\dfrac{|F(x,y)-F(z,y)|^p}{|z-x|^{n+sp}}
			 |z-x|^{n+sp} \\
			 &\quad +2^{p-1}
			 \dfrac{|F(z,y)-F(z,w)|^p}{|w-y|^{m+sp}}
			 |w-y|^{n+sp}\\
			 &\le 2^{p-1}r^{n+sp}\dfrac{|F(x,y)-F(z,y)|^p}{|z-x|^{n+sp}}\\
			 &\quad +2^{p-1}r^{m+sp}\dfrac{|F(z,y)-F(z,w)|^p}{|w-y|^{m+sp}}\\
			 &= 2^{p-1}r^{n+sp}
			 \left|\log\left(
			 \dfrac{u(x,y)+d}{u(z,y)+d}\right)
			 \right|^p\dfrac{1}{|z-x|^{n+sp}}\\
			 &\quad +2^{p-1}r^{m+sp}\left|\log\left(
			 \dfrac{u(z,y)+d}{u(z,w)+d}\right)
			 \right|^p\dfrac{1}{|w-y|^{m+sp}}.
		\end{align*}
		Therefore,
		\begin{align*}
			|Z|&|F_d(x,y)|^p=\iint_{Z}|F_d(x,y)|^p dwdz\\
			 &\le c_1r^{n+m+sp}
			 \int_{B^n_j }\left|\log\left(
			 \dfrac{u(x,y)+d}{u(z,y)+d}\right)
			 \right|^p\dfrac{dz}{|z-x|^{n+sp}}\\
			 &\quad +2^{p-1}r^{m+sp}
			 \int_{B^n_j}\int_{B^m_j}
			 \left|\log\left(
			 \dfrac{u(z,y)+d}{u(z,w)+d}\right)
			 \right|^p\dfrac{dwdz}{|w-y|^{m+sp}}
		\end{align*}
		for any $(x,y)\in B^n(x_j,\nicefrac{r}2)
		\times B^m(y_j,\nicefrac{r}2).$ Here $c_1=c_1(m,p)>0$ is a constant.
		Then
		\begin{align*}
			\int_{B^n_j}\int_{B^m_j}
			&|F_d(x,y)|^p dxdy\\
			 &\le \dfrac{c_1r^{n+m+sp}}{|Z|}
			 \int_{B^m_j}\int_{B^n_j}
			 \int_{B^n_j}\left|\log\left(
			 \dfrac{u(x,y)+d}{u(z,y)+d}\right)
			 \right|^p\dfrac{dzdxdy}{|z-x|^{n+sp}}\\
			 &\quad +\dfrac{c_2r^{n+m+sp}}{|Z|}
			 \int_{B^n_j}
			 \int_{B^m_j}
			 \int_{B^m_j}
			 \left|\log\left(
			 \dfrac{u(x,y)+d}{u(x,w)+d}\right)
			 \right|^p\dfrac{dwdxdy}{|w-y|^{m+sp}}.
		\end{align*}
		Thus, by Lemma \ref{lema:DKP} and since $u\ge 0 $ in $\Omega,$ we
		get
		\[
			\int_{B^n_j}\int_{B^m_j}
			|F_d(x,y)|^p dxdy\le C\dfrac{r^{2n+2m}}{|Z|},
		\]
		where $C=C(n,m,s,p)>0$ is a constant. Taking $d \to 0$ in the last
		inequality, we get that $u\equiv0$ in $B_j^n\times B_j^m.$
		
		By \eqref{eq:inc_int}, there exists $i\in\{1,\dots,k\}$ such
		that $i\neq j$ and
		\[
			|(B^n_i\times B^m_i)\cap\{(x,y)\colon u(x,y)=0\}|>0.
		\]
		Then, we can repeat the previous argument for
		$B_i^n\times B_i^m$ and obtain $u\equiv0$ in $B_i^n\times B_i^m.$
		In this way we conclude that $u\equiv 0$ in $K$ which 
		contradicts the fact
		that $u\not\equiv0$ in $K.$ Thus
		$|\{(x,y)\colon u(x,y)=0\}\cap K|=0.$
	\end{proof}

	Now, we are ready to prove that the eigenfunctions associated to
	$\lambda_1(s,p)$ do not change sign.
	\begin{teo}\label{teo:autposi}
		Assume {\rm{A1}} and {\rm{A2}}. If
		$u$ is an eigenfunction  associated to
		$\lambda_1(s,p)$ then $|u|>0$ in $\Omega.$
	\end{teo}
	
	\begin{proof}
		We start by showing that if $u$ is an eigenfunction corresponding to 
		$\lambda_1(s,p)$ then $|u|\not\equiv0$ in all connected components of 
		$\Omega.$ Our proof is by contradiction. We therefore assume that
		there is a connected component $A$ of $\Omega$ such that $|u|\equiv0.$
		Since $u$ is an eigenfunction corresponding to $\lambda_1(s,p)$
		then so is $|u|.$ Then
		\[
			\begin{aligned}
				0&=\lambda_1(s,p)\int_{\Omega}|u(x,y)|^{p-1}\phi(x,y)\, dxdy
				=\mathcal{H}_{s,p}(|u|,\phi)\\
				&=-2\int_{A^c}\int_{A_y}
				\dfrac{|u(x,y)|^{p-1}\phi(z,y)}{|x-z|^{n+sp}}dzdxdy
				-2\int_{A^c}\int_{A_x}
				\dfrac{|u(x,y)|^{p-1}\phi(x,w)}{|y-w|^{m+sp}}dwdxdy
			\end{aligned}
		\]
		for all $\phi\in C^\infty_0(A),$ which is a contradiction.

		Therefore, if $A$ connected components $C$ of 
		$\Omega$  then $|u|\not\equiv 0$ 
		in $A$ and
		\[
   			\mathcal{H}_{s,p}(|u|,v)=\lambda_1(s,p)
   			\int_{\Omega}|u(x,y)|^{p-1}v(x,y)\, dxdy\ge0\quad\forall v\in
   			\widetilde{\mathcal{W}}^{s,p}(A).
   		\]
   		Then, by Lemma \ref{lema:positivo}, $|u|>0$ in $A.$ Therefore
   		$|u|>0$ in $\Omega.$
 	\end{proof}
 	
 	 Our next result show that $\lambda_1(s,p)$ is simple.
 	
 	\begin{teo}\label{teo:autoval1}
		 Assume {\rm A1} and {\rm A2}.
    	 Let $u$ be a positive eigenfunction corresponding to
    	 $\lambda_{1}(s,p).$ If $\lambda>0$ is such that there exists
    	 a non-negative eigenfunction $v$ of \eqref{eq:EDP} with
    	 eigenvalue $\lambda,$ then $\lambda=\lambda_1(s,p)$ and
    	 there exists $k\in\mathbb{R}$ such that $v = k u$ a.e. in $\Omega.$
	\end{teo}
	
	\begin{proof}
		Since $\lambda_1(s,p)$ is the first eigenvalue we have
		that $\lambda_1(s,p)\le\lambda$.
		Let $k\in\mathbb{N}$ and define $v_k\coloneqq v+\nicefrac1{k}.$

		We begin proving that
		$w_{k}\coloneqq u^{p} / v_k^{p-1}\in
		\widetilde{\mathcal{W}}^{s,p}(\Omega).$
		It is immediate that  $w_k=0$ in $\Omega^c$
		and $w_{k}\in L^{p}(\Omega),$ due to the fact that
		$u\in L^{\infty}(\Omega),$ see
	 	Lemma \ref{lema:cotainfty}.

		On the other hand
		\begin{align*}
			|w_{k}(x,y)&-w_{k}(z,w)|
			\\
			=& \Bigg|
			\dfrac{u(x,y)^{p}-u(z,w)^p}{v_k(x,y)^{p-1}}
			+\dfrac{u(z,w)^p\left(v_k(z,w)^{p-1}-v_k(x,y)^{p-1}\right)}
			{v_k(x,y)^{p-1}v_k(z,w)^{p-1}}\Bigg|\\
			\le& k^{p-1}\left|u(x,y)^{p}-u(z,w)^p\right|
			+\|u\|_{L^{\infty}(\Omega)}^p
			\dfrac{\left|v_k(x,y)^{p-1}-v_k(z,w)^{p-1}\right|}
			{v_k(x,y)^{p-1}v_k(w,z)^{p-1}}\\
			\le&2\|u\|_{L^{\infty}(\Omega)}^{p-1}k^{p-1}p
			|u(x,y)-u(z,w)|\\
			&+\|u\|_{L^{\infty}(\Omega)}^p(p-1)
			\dfrac{v_k(x,y)^{p-2}+v_k(z,w)^{p-2}}
			{v_k(x,y)^{p-1}v_k(z,w)^{p-1}}|v_k(x,y)-v_k(z,w)|\\
			\le& 2\|u\|_{L^{\infty}(\Omega)}^{p-1}k^{p-1}p|u(x,y)-u(z,w)|\\
			&+\|u\|_{L^{\infty}(\Omega)}^p(p-1)k^{p-1}
			\left(\dfrac1{v_k(x,y)}+\dfrac1{v_k(z,w)}\right)|v(y)-v(x)|\\
	 		\le& C(k,p,\|u\|_{L^{\infty}(\Omega)})
			\left(|u(x,y)-u(z,w)|+|v(x,y)-v(z,w)|\right)
		\end{align*}
		for all $(x,y),(z,w)\in\mathbb{R}^{n+m}.$ Hence, we have that
		$w_{k}\in \widetilde{\mathcal{W}}^{s,p}(\Omega)$
		for all $k\in\mathbb{N}$ since $u,v\in
		\widetilde{\mathcal{W}}^{s,p}(\Omega).$
		
		Set
		\begin{align*}
			L(u,v_k)(x,y,z,w)&=
	   		|u(x,y)-u(w,z)|^p\\& \quad -(v_k(x,y)-v_k(w,z))^{p-1}
	   		\left(\dfrac{u(x,y)^p}{v_k(x,y)^{p-1}}
	   		-
	   		\dfrac{u(z,w)^p}{v_k(z,w)^{p-1}}\right).
		\end{align*}
		Then, by  \cite[Lemma 6.2]{A} and since
		$u,v$ are two positive eigenfunctions of
		problem \eqref{eq:EDP} with eigenvalues $\lambda_1(s,p)$ and
		$\lambda$ respectively, we have
		\begin{align*}
    		0\le&
    		\int_{\mathbb{R}^{n+m}}\int_{\mathbb{R}^n}
    		\dfrac{L(u,v_k)(x,y,z,y)}{|x-z|^{n+sp}} dzdxdy
    		+\int_{\mathbb{R}^{n+m}}\int_{\mathbb{R}^m}
    		\dfrac{L(u,v_k)(x,y,x,w)}{|y-w|^{m+sp}} dwdxdy\\
    		\le& \int_{\mathbb{R}^{n+m}}\int_{\mathbb{R}^n}
    		\dfrac{|u(x,y)-u(z,y)|^p}{|x-z|^{n+sp}} dz dxdy
    		+\int_{\mathbb{R}^{n+m}}\int_{\mathbb{R}^m}
    		\dfrac{|u(x,y)-u(x,w)|^p}{|y-w|^{n+sp}} dw dxdy\\
    		&-\mathcal{H}_{s,p}(v,w_k)\\
			\le&
			\lambda_{1}(s,p)\int_{\Omega}u(x,y)^p\,dxdy-
			\lambda\int_{\Omega}v(x,y)^{p-1}
			w_k(x,y)\, dxdy\\
			=&
			\lambda_{1}(s,p)\int_{\Omega}u(x,y)^p\,dxdy-
			\lambda\int_{\Omega}v(x,y)^{p-1}
			\dfrac{u(x,y)^p}{v_k(x,y)^{p-1}}\, dxdy.
		\end{align*}
		By Fatou's lemma and the dominated convergence theorem we obtain
		\[
    		\int_{\mathbb{R}^{n+m}}\int_{\mathbb{R}^n}
    		\dfrac{L(u,v)(x,y,z,y)}{|x-z|^{n+sp}} dzdxdy
			+\int_{\mathbb{R}^{n+m}}\int_{\mathbb{R}^m}
			\dfrac{L(u,v)(x,y,x,w)}{|y-w|^{m+sp}} dwdxdy=0
		\]
		due to $\lambda_1(s,p,h)\le \lambda.$
		Then $L(u,v)(x,y,z,y)=L(u,v)(x,y,x,w)=0$ a.e.
		Hence, again by Lemma 6.2 in \cite{A},
		$u(x,y)=\ell_1(y)v(x,y)$ and $u(x,y)=\ell_2(x)v(x,y)$
		for all $(x,y)\in \mathbb{R}^{n+m}$. Then, we conclude that $u=\ell v$
		for some constant $\ell>0.$
	\end{proof}
	
	Finally we will prove that $\lambda_1(s,p)$ is isolated.
	
	\begin{teo}\label{teo:autposi.2}
		Assume {\rm{A1}} and {\rm{A2}}. Them $\lambda_1(s,p)$ is isolated.
	\end{teo}
	\begin{proof}
		We split the proof into two steps.
		
		{\noindent \it Step 1.} If $u$ is an eigenfunction associated to some 
		eigenvalue $\lambda > \lambda_1(s,p)$ then there is a positive constant $C$ such that
		\begin{equation} \label{macri}
			\left(\dfrac{1}{C\lambda}\right)^{\nicefrac{r}{(r-p)}}
			\le|\Omega_{\pm}|
		\end{equation}
		for all $p<r<p_s^\star.$ Here $\Omega_{\pm}=\{(x,y)\colon
		u_{\pm}\not\equiv0\},$ and
		\[
			p_s^\star= 
				\begin{cases}
						\dfrac{(n+m)p}{n+m-sp}, &\text{ if }sp<n+m,\\
						\infty &\text{ if } sp\ge n+m.
				\end{cases}
		\]
		Let $r\in (p,p_s^\star).$
		By Theorem \ref{teo:compacemb}, Lemmas \ref{lema:Poicare} 
		and  \ref{lema:contincl} and H\"older inequality, we have
		\begin{align*}
			\|u_+\|_{L^r(\Omega)}^p\le C\|u_+\|_{\wsp}^p
			\le C\mathcal{H}_{s,p}(u,u_+)=C\lambda
			\|u_+\|_{L^r(\Omega)}^p|\Omega_{+}|^{\nicefrac{(r-p)}{r}}.
		\end{align*} 
		Then
		\[
			\left(\dfrac{1}{C\lambda}\right)^{\nicefrac{r}{(r-p)}}
			\le|\Omega_{+}|.
		\]
		
		In order to prove the inequality for $|\Omega_{-}|$, 
		it suffices to proceed as above, using the function 
		$-u$ instead of $u.$
		
		{\noindent \it Step 2.} By definition, $\lambda_1(s,p)$
		is left-isolated. To prove that $\lambda_1(s,p)$
		is right-isolated, we argue by contradiction. We assume 
		that there is a sequence of eigenvalues 
		$\{\lambda_k\}_{k\in \mathbb{N}}$ such that 
		$\lambda_k \searrow \lambda_1(s,p)$ as $k\to \infty.$
		Let $u_k$ be an eigenfunction associated to $\lambda_k$ such
		that $\|u_k\|_{\lp}=1.$ Then $\{u_k\}_{k\in \mathbb{N}}$ is
		bounded in $\widetilde{\mathcal{W}}^{s,p}(\Omega)$ and therefore
		we can extract a subsequence (that we still denoted by 
		$\{u_k\}_{k\in \mathbb{N}}$) such that 
		\[
				u_k\rightharpoonup u  \text{ weakly in }
				\widetilde{\mathcal{W}}^{s,p}(\Omega),\qquad
				u_k\to u  \text{ strongly in }
				\lp.
		\]	
		Then $\|u\|_{\lp}=1$ and 
		\[
			[u]^p_{\wwspr}\le\liminf_{k\to\infty}[u_k]^p_{\wwspr}
			=\lim_{k\to\infty}\lambda_k=\lambda_1(s,p).
		\]
		Then $u$ is an eigenfunction associated to $\lambda_1(s,p).$ 
		Therefore $u$ has constant sign.
		
		Now, proceeding as in the proof of \cite[Theorem 2]{Anane},
		we arrive to a contradiction. In fact, by Egoroff's theorem we can find a subset $A_\delta$ 
		of $\Omega$ such that $|A_\delta| < \delta$ and $u_k \to u$ uniformly in $\Omega \setminus A_\delta$. From \eqref{macri} we get that $u$ and the uniform convergence in  $\Omega \setminus A_\delta$ we obtain that $|\{u>0\}| >0$ and $|\{u>0\}| <0$.
		 This contradicts the fact that an eigenfunction associated with the first eigenvalue does not change sign. 
	\end{proof}

\section{The limit as $s\to 1^-$}\label{limites}
	In this section, our goal is to show that
	\begin{equation}\label{eq:sa1}
		\begin{aligned}
			&\lim_{s\to 1^-}(1-s)\lambda_1(s,p)= \lambda_1(1,p)\\
			&=\inf_{ u\in W^{1,p}_0(\Omega), u\not\equiv0}\left\{
			\dfrac{K_{n,p}\displaystyle\int_{\Omega}|\nabla_x u(x,y)|^p dxdy
			+K_{m,p}\int_{\Omega}|\nabla_y u(x,y)|^p dxdy}{\displaystyle
			\|u\|_{\lpp}^p}\right\}
		\end{aligned}
	\end{equation}
	where $K_{n,p}$ is a constant that depends only on $n$ and $p,$ and  
	$K_{m,p}$ depends only on $m$ and $p.$ 
	Before proving \eqref{eq:sa1}, we need some technical results. 
	
	\begin{lem}\label{lema:inclu3}
		Let $\Omega$ be an open subsets of 
	 	$\mathbb{R}^{n+m}$ with smooth boundary and $p\in(1,\infty).$ 
	 	For all $s\in(0,1)$ we have that
		$W^{1,p}(\Omega)$ is continuity embedded  in 
		$\wwsp.$ 
	\end{lem}
	\begin{proof}
		In this proof, we follow the ideas of the proof of 
		\cite[Theorem 1]{Bourgain}.
		Let $u\in W^{1,p}(\Omega).$
		By an extension argument, we can assume that 
		$u\in W^{1,p}(\mathbb{R}^{n+m}).$ 
		We have that
		\begin{equation}\label{eq:inclu31}
			\begin{aligned}
				\int_{\mathbb{R}^{n+m}} |u(x+h,y)-u(x,y)|^p dxdy&\le
				|h|^p\int_{\mathbb{R}^{n+m}} |\nabla_x u(x,y)|^p dxdy,\\
				\int_{\mathbb{R}^{n+m}} |u(x,y+h)-u(x,y)|^p dxdy&\le
				|h|^p\int_{\mathbb{R}^{n+m}} |\nabla_y u(x,y)|^p dxdy.
			\end{aligned}
		\end{equation}
		The proof of this fact can be carried out as that of Proposition XI.3 in 
		\cite{Brezis} and is omitted.	
		
		Then, by \eqref{eq:inclu31}, we have
		\begin{align*}
			\int_{\mathbb{R}^n}&\int_{\mathbb{R}^{n+m}}
			\dfrac{|u(x,y)-u(z,y)|^p}{|x-z|^{n+sp}} dxdydz\\
			&=\int_{\mathbb{R}^n}\int_{\mathbb{R}^{n+m}}
			\dfrac{|u(x+h,y)-u(x,y)|^p}{|h|^{n+sp}} dxdydh\\
			&\le\int_{\{|h|\le 1\}}\dfrac{dh}{|h|^{(s-1)p+n}} 
			\int_{\mathbb{R}^{n+m}} |\nabla_x u(x,y)|^p dxdy\\
			&\qquad+2\int_{\{|h|> 1\}}
			\dfrac{dh}{|h|^{sp+n}}
			\int_{\mathbb{R}^{n+m}} |u(x,y)|^p dxdy\\
			&\le\dfrac{\omega_n}{(1-s)p}
			\int_{\mathbb{R}^{n+m}} |\nabla_x u(x,y)|^p dxdy
			+\dfrac{2\omega_n}{sp}\int_{\mathbb{R}^{n+m}} |u(x,y)|^p dxdy.
		\end{align*}
		Similarly,
		\[
			\begin{aligned}
				\int_{\mathbb{R}^m}\int_{\mathbb{R}^{n+m}}&
			\dfrac{|u(x,y)-u(x,w)|^p}{|y-w|^{m+sp}} dxdydw\\
			&\le\dfrac{\omega_m}{(1-s)p}
			\int_{\mathbb{R}^{n+m}} |\nabla_y u(x,y)|^p dxdy
			+\dfrac{2\omega_m}{sp}\int_{\mathbb{R}^{n+m}} |u(x,y)|^p dxdy,
			\end{aligned}
		\]
		which completes the proof.
	\end{proof}	
	
	\begin{remark}\label{re:inclu3}
		Proceeding as in the proof of previous lemma 
		and using using the Poincar\'e 
		inequality, we have that
		\[
			(1-s)[u]_{\mathcal{W}^{s,p}(
						\mathbb{R}^{n+m})}^p
						\le C \left(1+\dfrac{1}{s}\right)
						\int_{\Omega}|\nabla u|^p \, dxdy
						\qquad\forall u \in W^{1,p}_0(\Omega)
		\]
		where $C$ is a constant independent of $s.$
	\end{remark}
	
	\begin{lem}\label{lema:limites}
		Let $\Omega$ be an open subset of 
	 	$\mathbb{R}^{n+m}$ with smooth boundary and $p\in(1,\infty).$ 
	 	If $u\in W^{1,p}_0(\Omega)$ then
	 	 \[
	 	 	(1-s)[u]_{\mathcal{W}^{s,p}(
						\mathbb{R}^{n+m})}^p\to
						K_{n,p}\int_{\Omega}|\nabla_x u|^p\, dxdy+
						K_{m,p}\int_{\Omega}|\nabla_y u|^p\, dxdy
	 	 \]
	 	 as $s\to 1^-.$
	\end{lem}
	\begin{proof}
		We split the proof into two cases.
		
		{\noindent \it Case 1.} 
		First we prove the lemma for $\phi\in C_0^{\infty}(\Omega).$
		Let $B_1$ and $B_2$ be two open balls in $\mathbb{R}^n$ and 
		$\mathbb{R}^m$ respectively such that $\Omega\subset B_1\times B_2.$
		
		Given $y\in B_2,$ we have that
		\begin{equation}\label{eq:limites1}
			\begin{aligned}
				\int_{\mathbb{R}^n}\int_{\mathbb{R}^n}
				\dfrac{|\phi(x,y)-\phi(z,y)|^p}{|x-z|^{n+sp}}	dxdz&=
				\int_{B_1}\int_{B_1}
				\dfrac{|\phi(x,y)-\phi(z,y)|^p}{|x-z|^{n+sp}}	dxdz\\
				&\quad +2\int_{B_1}\int_{B_1^c}
				\dfrac{|\phi(x,y)|^p}{|x-z|^{n+sp}}	dxdz.
			\end{aligned} 
		\end{equation}
		By \cite[Theorem 1]{Bourgain}, there is a constant $K_{n,p}$
		(that depends only the $n$ and $p$) such that
		\begin{equation}\label{eq:limites2}
				(1-s)\int_{B_1}\int_{B_1}
				\dfrac{|\phi(x,y)-\phi(z,y)|^p}{|x-z|^{n+sp}}	dxdz
				\to K_{n,p}
				\int_{B_1} |\nabla_x \phi(x,y)|^p dx
		\end{equation}
		as $s\to1^-.$
		On the other hand, since 
		$\supp(\varphi)\subset\subset \Omega\subset B_1\times B_2,$ 
		there exists $\delta>0$ such that $|x-z|>\delta$ 
		for all $z\in B_1^c$ and 
		$x\in\{t\in B_1\colon (t,y)\in \supp(\varphi)\}.$ Thus
		\begin{equation}\label{eq:limites3}
			(1-s)\int_{B_1}\int_{B_1^c}
				\dfrac{|\phi(x,y)|^p}{|x-z|^{n+sp}}	dxdz
				\le(1-s)\dfrac{\omega_n}{sp\delta^{sp}}
				\|\phi(\cdot,y)\|_{L^p(B_1)}^p\to0
		\end{equation}
		as $s\to 1^{-}.$ Then by \eqref{eq:limites1},
		\eqref{eq:limites2}, and  \eqref{eq:limites3} we have that 
		\begin{equation}\label{eq:limites4}
				(1-s)\int_{\mathbb{R}^n}\int_{\mathbb{R}^n}
				\dfrac{|\phi(x,y)-\phi(z,y)|^p}{|x-z|^{n+sp}}	dxdz
				\to K_{n,p}
				\int_{B_1} |\nabla_x \phi(x,y)|^p dx
		\end{equation}
		as $s\to 1^-.$
		Proceeding as in the proof of Lemma \ref{lema:inclu3}, we have that
		\begin{align*}
			(1-s)\int_{\mathbb{R}^n}\int_{\mathbb{R}^n}
				\dfrac{|\phi(x,y)-\phi(z,y)|^p}{|x-z|^{n+sp}}	dxdz&\le
				\dfrac{\omega_n}{p}
			\int_{\mathbb{R}^{n}} |\nabla_x \phi(x,y)|^p dxdy\\
			&+(1-s)
			\dfrac{2\omega_n}{s_0p}\int_{\mathbb{R}^{n}} 
			|\phi(x,y)|^p dxdy.
		\end{align*}
		 Thus, \eqref{eq:limites4} and the dominated convergence theorem imply
		 \[
				(1-s)\int_{\mathbb{R}^{n+m}}\int_{\mathbb{R}^n}
				\dfrac{|\phi(x,y)-\phi(z,y)|^p}{|x-z|^{n+sp}}dzdxdy
				\to K_{n,p}
				\int_{\mathbb{R}^m}\int_{B_1} |\nabla_x \phi(x,y)|^p dxdy,
		\]
		as $s\to 1^-,$ that is,
		\begin{equation}\label{eq:limites5}
				(1-s)\int_{\mathbb{R}^{n+m}}\int_{\mathbb{R}^n}
				\dfrac{|\phi(x,y)-\phi(z,y)|^p}{|x-z|^{n+sp}}dzdxdy
				\to K_{n,p}
				\int_{\Omega} |\nabla_x \phi(x,y)|^p dxdy,
		\end{equation}  
		as $s\to1^-.$
		
		In the same manner we can see that 
		there exists a constant $K_{m,p}$ (that depends only the $m$ and $p$) 
		such that
		\begin{equation}\label{eq:limites6}
				(1-s)\int_{\mathbb{R}^{n+m}}\int_{\mathbb{R}^m}
				\dfrac{|\phi(x,y)-\phi(x,w)|^p}{|y-w|^{m+sp}}dwdxdy
				\to K_{m,p}
				\int_{\Omega} |\nabla_y \phi(x,y)|^p dxdy,
		\end{equation}  
		as $s\to1^-.$ 
		
		Then, we have
		\[
	 	 	(1-s)[\phi]_{\mathcal{W}^{s,p}(
						\mathbb{R}^{n+m})}^p\to
						K_{n,p}\int_{\Omega}|\nabla_x \phi|^p\, dxdy+
						K_{m,p}\int_{\Omega}|\nabla_y \phi|^p\, dxdy,
	 	 \]
	 	 as $s\to 1^-.$
	 	 
	 	 {\noindent \it Case 2.} 
	 	 Now we prove the general case. Given $u\in W^{1,p}_0(\Omega),$
	 	 we define  
	 	 \[
	 	 	\begin{aligned}
					F^u_s(x,y,z)&=(1-s)^{\nicefrac1p}\dfrac{|u(x,y)-u(z,y)|}
	 	 					{|x-z|^{\nicefrac{n}p+s}},\\
	 	 			G^u_s(x,y,z)&=(1-s)^{\nicefrac1p}\dfrac{|u(x,y)-u(x,w)|}
	 	 					{|y-w|^{\nicefrac{m}p+s}}
			\end{aligned}
	 	 \]	
	 	 and we want to show that
	 	 \[
	 	 	\begin{aligned}
	 	 		\|F_s^u\|_{L^p(\mathbb{R}^{2n+m})}\to K_{n,p}^{\nicefrac1p}
	 	 		\|\nabla_x u\|_{\lpp},\qquad
	 	 		\|G_s^u\|_{L^p(\mathbb{R}^{n+2m})}\to K_{m,p}^{\nicefrac1p}
	 	 		\|\nabla_y u\|_{\lpp},
	 	 	\end{aligned}
	 	 \]
	 	 as $s\to 1^-.$
	 	 
	 	 Given $\varepsilon>0$ there is
	 	 $\phi\in C^{\infty}_0(\Omega)$ such that
	 	 \[
	 	 	\|\nabla u-\nabla \phi\|_{L^p(\Omega)}<\varepsilon.
	 	 \]
	 	 Thus
	 	 \begin{equation}\label{eq:limites7}
			|\|\nabla_x u\|_{\lpp} - \|\nabla_x \phi \|_{\lp}|
	 	 	<\varepsilon \text{ and } 
	 	 	|\|\nabla_x u\|_{\lpp} - \|\nabla_x \phi \|_{\lp}|
	 	 	<\varepsilon.
		\end{equation}
	 	 
	 	 By case 1, there exists $s_0\in(0,1)$ such that
	 	 \begin{equation}\label{eq:limites8}
			\begin{aligned}
				&|\|F_s^\phi\|_{L^p(\mathbb{R}^{2n+m})}
				-K_{n,p}^{\nicefrac1p}
	 	 		\|\nabla_x \phi\|_{\lp}| < \varepsilon,\\
	 	 		&|\|G_s^\phi\|_{L^p(\mathbb{R}^{n+2m})}
	 	 		-K_{m,p}^{\nicefrac1p}
	 	 		\|\nabla_y \phi\|_{\lp}|<\varepsilon,
			\end{aligned}	
		\end{equation}
	 	 for all $s\in(s_0,1).$
	 	 
	 	 On the other hand, using  Remark \ref{re:inclu3}, we have that
	 	 \begin{equation}\label{eq:limites9}
			\begin{aligned}
				|\|F_s^u\|_{L^p(\mathbb{R}^{2n+m})}-
				\|F_s^\phi\|_{L^p(\mathbb{R}^{2n+m})}|	
				\le C \|\nabla u -\nabla \phi \|_{L^p(\Omega)}
				<C\varepsilon,\\
				|\|G_s^u\|_{L^p(\mathbb{R}^{2n+m})}-
				\|G_s^\phi\|_{L^p(\mathbb{R}^{2n+m})}|	
				\le C \|\nabla u -\nabla \phi \|_{L^p(\Omega)},
				<C\varepsilon,
			\end{aligned}
		\end{equation}
		where $C$ is a constant independent of $s.$
		
		Finally, by \eqref{eq:limites7}, \eqref{eq:limites8}, and
		\eqref{eq:limites9}, we obtain that
		\[
			\begin{aligned}
				&|\|F_s^u\|_{L^p(\mathbb{R}^{2n+m})}
				-K_{n,p}^{\nicefrac1p}
	 	 		\|\nabla_x u\|_{\lpp}| <C\varepsilon,\\
	 	 		&|\|G_s^u\|_{L^p(\mathbb{R}^{n+2m})}
	 	 		-K_{m,p}^{\nicefrac1p}
	 	 		\|\nabla_y u\|_{\lpp}|<C\varepsilon,
			\end{aligned}	
		\]
		and the proof is complete.
	\end{proof}
	
	\begin{cor}\label{cor:limites}
		Let $\Omega$ be an open subset of 
	 	$\mathbb{R}^{n+m}$ with smooth boundary and $p\in(1,\infty).$ 
	 	If $u\in W^{1,p}_0(\Omega)$ then
	 	 \[
	 	 	(1-s)[u]_{\mathcal{W}^{s,p}(\Omega)}^p\to
						K_{n,p}\int_{\Omega}|\nabla_x u|^p\, dxdy+
						K_{m,p}\int_{\Omega}|\nabla_y u|^p\, dxdy
	 	 \]
	 	 as $s\to 1^-.$	
	\end{cor}
	\begin{proof}
		By Lemma \ref{lema:limites}, we only need to proof that if 
		$u\in W^{1,p}_0(\Omega)$ then
		\[
	 	 	(1-s)\left([u]_{\mathcal{W}^{s,p}(\mathbb{R}^{n+m})}^p
	 	 	-[u]_{\mathcal{W}^{s,p}(\Omega)}^p\right)\to
						0
	 	 \]
	 	 as $s\to 1^-.$	
		First we prove the result for $\phi\in C_0^{\infty}(\Omega).$
		We have
		\begin{equation}\label{eq:limites10}
			\begin{aligned}
				\left([\phi]_{\mathcal{W}^{s,p}(\mathbb{R}^{n+m})}^p
		 	 	-[\phi]_{\mathcal{W}^{s,p}(\Omega)}^p\right)&=
		 	 	2\int_{\supp(\phi)}\int_{\Omega_y^c}
		 	 	\dfrac{|\phi(x,y)|}{|x-z|^{n+sp}}^p dzdxdy\\
		 	 	&\qquad+
		 	 	2\int_{\supp(\phi)}\int_{\Omega_x^c}
		 	 	\dfrac{|\phi(x,y)|}{|y-w|^{m+sp}}^p dwdxdy.
			\end{aligned}
		\end{equation}
		Since $\supp(\phi)\subset\Omega$ is compact, there exists 
		$\delta>0$ such that $|x-z|>\delta$ and $|y-w|>\delta$ for all
		$(x,y)\in \supp(\phi),$  $z\in\Omega_y^c,$ $w\in \Omega_x^c.$ 	
		Then
		\begin{align*}
			\int_{\supp(\phi)}\int_{\Omega_y^c}
	 	 	\dfrac{|\phi(x,y)|}{|x-z|^{n+sp}}^p dzdxdy
	 	 	&\le \dfrac{\omega_n}{sp\delta^{sp}}\int_{\Omega}
	 	 	|\phi(x,y)|^pdxdy,
	 	 	\\
	 	 	\int_{\supp(\phi)}\int_{\Omega_y^c}
	 	 	\dfrac{|\phi(x,y)|}{|y-w|^{m+sp}}^p dwdxdy
	 	 	&\le \dfrac{\omega_m}{sp\delta^{sp}}\int_{\Omega}
	 	 	|\phi(x,y)|^pdxdy.
		\end{align*}
		Therefore, using \eqref{eq:limites10}, we have that  
		\[
	 	 	(1-s)\left([\phi]_{\mathcal{W}^{s,p}(\mathbb{R}^{n+m})}^p
	 	 	-[\phi]_{\mathcal{W}^{s,p}(\Omega)}^p\right)\to
						0
	 	 \]
	 	 as $s\to 1^-.$	
	 	 
	 	 The argument for the general case is analogous to the one performed in case 2 in the proof of 
	 	 Lemma \ref{lema:limites}. 
	\end{proof}
	
	For the proof of the following lemma, see  \cite[Lemma 2]{Bourgain}.
	\begin{lem}\label{lema:Blema2}
		Let $\delta>0$ and $g,h\colon(0,\delta)\to(0,+\infty).$
		Assume that $g(t)\le g(\nicefrac{t}2)$ and that $h$ in non-increasing.
		Then
		\[
			\int_{0}^\delta t^{N-1}g(t)h(t)\, dt\ge\dfrac{N}{(2\delta)^{N}}
			\int_{0}^{\delta}t^{N-1}g(t) dt\int_0^{\delta}t^{N-1}h(t) dt
		\]
		for all $N>0.$ 
	\end{lem}
	
	\begin{lem}\label{lema:incluentres}
		Let $0<s_0<s$ and $u\in\widetilde{\mathcal{W}}^{s,p}(\Omega).$
		Then
		\[
			\dfrac{(1-s_0)[u]_{\mathcal{W}^{s_0,p}(\Omega)}^p}
			{2^{(1-s_0)p}\diam(\Omega)^{(s-s_0)p}}
			\le(1-s)[u]_{\mathcal{W}^{s,p}(\mathbb{R}^{n+m})}^p
		\]
	\end{lem}
	\begin{proof}
		Let $B_1$ and $B_2$ be two balls in $\mathbb{R}^n$ and $\mathbb{R}^m$	
		respectively such that $\Omega\subset B_1\times B_2$ and
		$\diam(B_1)=\diam(B_2)=\diam(\Omega).$ Then
		\begin{align*}
			&\int_{\mathbb{R}^{n+m}}\int_{\mathbb{R}^n}
			\dfrac{|u(x,y)-u(z,y)|^{p}}{|x-z|^{n+sp}} dzdxdy \ge\\
			&\ge\int_{\mathbb{R}^{m}}
			\int_0^{\infty}\int_{S^{n-1}}\int_{\mathbb{R}^n}
			\dfrac{|u(x+tw,y)-u(x,y)|^{p}}{t^{1+sp}} dxd\sigma dt dy\\ 
			&\ge \int_{\mathbb{R}^{m}}
			\int_0^{\diam(\Omega)}\int_{S^{n-1}}
			t^{(1-s_0)p-1}
			\int_{\mathbb{R}^n}
			\dfrac{|u(x+t\omega,y)-u(x,y)|^{p}}{t^p}
			\dfrac{ dxd\sigma dt dy}{t^{(s-s_0)p}}\\ 
		\end{align*}
		Taking $N=(1-s_0)p,$ $\delta=\diam(\Omega),$ we get
		\[
			g(t)=\int_{S^{n-1}}
			\int_{\mathbb{R}^m}
			\dfrac{|u(x+t\omega,y)-u(x,y)|^{p}}{t^p} dxd\sigma, \qquad
			\mbox{ and } \qquad h(t)=\dfrac{1-s}{t^{(s-s_0)p}}.
		\]
		By Lemma \ref{lema:Blema2}, we have that
		\begin{align*}
			&(1-s)\int_{\mathbb{R}^{n+m}}\int_{\mathbb{R}^n}
			\dfrac{|u(x,y)-u(z,y)|^{p}}{|x-z|^{n+sp}} dzdxdy \ge\\
			&\ge
			\dfrac{(1-s_0)p}{2^{(1-s_0)p}\diam(\Omega)^{(1-s_0)p}}
			\int_{\mathbb{R}^{m}}
			\int_0^{\delta} t^{(1-s_0)p-1}g(t) dt
			\int_0^{\delta} t^{(1-s_0)p-1}h(t) dt\\ 
			&\ge
			\dfrac{(1-s_0)p}{2^{(1-s_0)p}\diam(\Omega)^{(1-s_0)p}}
			\int_{\mathbb{R}^{m}}
			\int_0^{\delta} t^{(1-s_0)p-1}g(t) dt
			\int_0^{\delta} (1-s)t^{(1-s)p-1}dt\\ 
			&\ge
			\dfrac{(1-s_0)}{2^{(1-s_0)p}\diam(\Omega)^{(s-s_0)p}}
			\int_{\mathbb{R}^{m}}
			\int_0^{\delta} 
			\int_{S^{n-1}}
			\int_{\mathbb{R}^m}
			\dfrac{|u(x+t\omega,y)-u(x,y)|^{p}}{t^{1+s_0p}} dxd\sigma dtdy\\
			&\ge \dfrac{(1-s_0)}{2^{(1-s_0)p}\diam(\Omega)^{(s-s_0)p}}
			\int_{\Omega}\int_{\Omega_y}
			\dfrac{|u(x,y)-u(z,y)|^{p}}{|x-z|^{n+s_0p}} dzdxdy. 
		\end{align*}
		
		Similarly
		\begin{align*}
			(1-s)\int_{\mathbb{R}^{n+m}}\int_{\mathbb{R}^n}&
			\dfrac{|u(x,y)-u(x,w)|^{p}}{|y-w|^{m+sp}} dzdxdy \ge\\
			&\ge \dfrac{(1-s_0)}{2^{(1-s_0)p}\diam(\Omega)^{(s-s_0)p}}
			\int_{\Omega}\int_{\Omega_x}
			\dfrac{|u(x,y)-u(z,y)|^{p}}{|y-w|^{m+s_0p}} dwdxdy. 
		\end{align*}
		This concludes the proof.
	\end{proof}
	We can now show the main result of this section.
	\begin{teo}\label{teo:limites}
		Let $\Omega$ is bounded domain in
		$\mathbb{R}^{n+m}$ with smooth boundary, $s\in(0,1)$ and 
		$p\in(1,\infty).$ Then
		\[
			\lim_{s\to 1^-}(1-s)\lambda_1(s,p)= \lambda_1(1,p).
		\]
	\end{teo}
	\begin{proof}
		First, we observe that, from Lemma \ref{lema:inclu3}, 
		if $u\in W^{1,p}_0(\Omega)$ then 
		$u\in \widetilde{\mathcal{W}}^{s,p}(\Omega).$ 
		Then
		\[
			(1-s)\lambda_1(s,p)\le\dfrac{[u]_{
			\mathcal{W}^{s,p}(\mathbb{R}^{n+m})}^p}
			{\|u\|_{L^p(\Omega)}^p}
		\]
		for all $u\in W^{1,p}_0(\Omega),$ $u\not\equiv0.$ 
		Therefore, by Lemma \ref{lema:limites}, we have that
		\[
			\limsup_{s\to 1^-}(1-s)\lambda_1(s,p)\le
			\dfrac{K_{n,p}\displaystyle\int_{\Omega}
			|\nabla_x u(x,y)|^p dxdy
			+K_{m,p}\int_{\Omega}|\nabla_y u(x,y)|^p dxdy}{\displaystyle
			\|u\|_{\lpp}^p}
		\]
		for all $u\in W^{1,p}_0(\Omega),$ $u\not\equiv0.$ Then
		\begin{equation}\label{eq:tlimites1}
				\limsup_{s\to 1^-}(1-s)\lambda_1(s,p)\le\lambda_1(1,p).
		\end{equation}
		
		 To finish the proof, we have to show that
		 \[
				\liminf_{s\to 1^-}(1-s)\lambda_1(s,p)\ge\lambda_1(1,p).
		 \]
		 
		 Let $\{s_k\}_{k\in \mathbb{N}}\subset(0,1)$ be such that
		 $s_k\to1$ as $k\to \infty,$
		 \begin{equation}\label{eq:tlimites2}
			\lim_{k\to \infty}
				(1-s_k)\lambda_1(s_k,p)
				=\liminf_{s\to 1^-}(1-s)\lambda_1(s,p).	
		\end{equation}
		For each $k\in \mathbb{N},$ we let $u_k$ be an 
		eigenfunction corresponding to $\lambda_1(s_k,p)$ such that
		$\|u_k\|_{L^p(\Omega)}=1.$ By \eqref{eq:tlimites2}, there is
		a positive constant $C$ such that
		\[
			(1-s_k)[u_k]_{\mathcal{W}^{s_k,p}(\mathbb{R}^{n+m})}^p\le C
			\qquad \forall k\in \mathbb{N}.
		\]
		Then, by Lemma \ref{lema:contincl}, there is
		a positive constant $C$ such that
		\[
			(1-s_k)|u_k|_{W^{s_k,p}(\mathbb{R}^{n+m})}^p\le C
			\qquad \forall k\in \mathbb{N}.
		\]
		Thus, by \cite[Corollary 7]{Bourgain}, up to a subsequence,
		$\{u_k\}_{k\in\mathbb{N}}$ converges in $\lp$ to some 
		$u\in W^{1,p}_{0}(\Omega).$ Moreover, for all $\delta>0$,
		$	u_k\to u \mbox{ strongly in } W^{1-\delta,p}(\Omega)$.
		Therefore $\|u\|_{\lp}=1.$
		
		Let $s_0\in(0,1).$ Since $s_k\to 1,$ there exists $k_0\in
		\mathbb{N}$ such that $s_0<s_k$ for all $k\ge k_0.$ Then,
		by Lemma \ref{lema:incluentres}, we have that
		\begin{align*}
			\dfrac{(1-s_0)[u_k]_{\mathcal{W}^{s_0,p}(\Omega)}^p}
			{2^{(1-s_0)p}}
			&\le
			\diam(\Omega)^{(s_k-s_0)p}(1-s_k)
			[u_k]_{\mathcal{W}^{s_k,p}(\mathbb{R}^n)}^p\\
			&=
			\diam(\Omega)^{(s_k-s_0)p}(1-s_k)
			\lambda_1(s_k,p).
		\end{align*}
		Thus, by \eqref{eq:tlimites2} and Fatou's lemma, we get
		\begin{equation}\label{eq:tlimites3}
			\dfrac{(1-s_0)[u]_{\mathcal{W}^{s_0,p}(\Omega)}^p}
			{2^{(1-s_0)p}}\le \diam(\Omega)^{(1-s_0)p}
			\liminf_{s\to 1^-}(1-s)\lambda_1(s,p).
		\end{equation}
		By Corollary \ref{cor:limites}, it holds that
		\begin{align*}
			K_{n,p}\displaystyle\int_{\Omega}
			|\nabla_x u(x,y)|^p dxdy
			+K_{m,p}\int_{\Omega}|\nabla_y u(x,y)|^p dxdy & =
			\lim_{s_0\to 1^-}
			\dfrac{(1-s_0)[u]_{\mathcal{W}^{s_0,p}(\Omega)}^p}
			{2^{(1-s_0)p}}\\
			&\le \liminf_{s\to 1^-}(1-s)\lambda_1(s,p).
		\end{align*}
		Then
		\[
			\lambda_1(1,p)\le \liminf_{s\to 1^-}(1-s)\lambda_1(s,p).
		\]
		Therefore, by \eqref{eq:tlimites1},
		\[
			\lambda_1(1,p)= \lim_{s\to 1^-}(1-s)\lambda_1(s,p),
		\]
		as we wanted to prove.
	\end{proof}

\section{The limit as $p\to \infty$}\label{limitep}

 	Now we want to pass to the limit as $p\to \infty$ in the first 
 	eigenvalue $\lambda_{1}(s,p)$.
	Our goal now is to show that 
	\[
		[\lambda_1(s,p)]^{\nicefrac1p}\to \Lambda_{\infty} (s) 
	\]
	where
	\[
		\Lambda_{\infty} (s)=\inf
		\left\{[u]_{\mathcal{W}^{s,\infty}(\mathbb{R}^{n+m})}\colon 
		u\in \mathcal{W}^{s,\infty}(\mathbb{R}^{n+m}),  
		\|u\|_{L^\infty (\Omega)}=1, u=0 \mbox{ in }\Omega^c\right\}.
	\]
	Observe that, by Arzela-Ascoli's theorem, the previous infimum is attained.

	We first prove a geometric characterization of $\Lambda_\infty (s)$.

	\begin{lem} \label{lema.Lambda.infty} Let
	$
		R_s= \displaystyle
		\max_{(x,y) \in \Omega} \displaystyle
		\min_{(z,w) \in \partial \Omega}  
		(|x-z|^{s}+|y-w|^{s}),
	$
	then
	$$
		\Lambda_\infty (s) =  \frac{1}{R_s}.
	$$
	\end{lem}

	\begin{proof}
	 	Let $u\in \mathcal{W}^{s,\infty}(\mathbb{R}^{n+m}),$  
		such that 
		$\|u\|_{L^\infty (\Omega)}=1,$ $u=0 \mbox{ in }\Omega^c$
		and
		$
			\Lambda_\infty (s) =[u]_{\mathcal{W}^{s,\infty}
			(\mathbb{R}^{n+m})}$.	
			Then, let $(x_0,y_0) \in \Omega$ be such that
	$
		u (x_0,y_0) =1. 
	$
	If $(z,w) \in \partial \Omega$ we have
	$$
		|u(x_0,y_0)-u(z,y_0)| \leq \Lambda_\infty (s) |x_0-z|^{s}
	$$
	and
	$$
		|u(z,y_0)-u(z,w)| \leq \Lambda_\infty (s) |y_0-w|^{s}.
	$$
	Then
	$$
		|u(x_0,y_0)-u(z,w)| 
		\leq \Lambda_\infty (s) (|x_0-z|^{s}+|y_0-w|^{s}).
	$$
	Therefore,
	$$
		1 \leq \Lambda_\infty (s) \min_{(z,w) 
		\in \partial \Omega} (|x_0-z|^{s}+|y_0-w|^{s}),
	$$
	and then, we get 
	\begin{equation} \label{ppp}
		\Lambda_\infty (s) \geq  
		\frac{1}{\displaystyle\min_{(z,w) \in \partial \Omega}  
		(|x_0-z|^{s}+|y_0-w|^{s})}
		\ge \frac{1}R_s.
	\end{equation}

	Now, we choose $(x_0,y_0)$ that solves the geometric maximization problem 
	$$
		R_s= \max_{(x,y) \in \Omega} \min_{(z,w) \in \partial \Omega}  
		(|x-z|^{s}+|y-w|^{s}),
	$$
	and consider the function
	$$
		u(x,y) = \left(1 - \frac{|x_0-x|^{s}+|y_0-y|^{s}}{R_s}\right)_+.
	$$
Observe that, $\|u\|_{L^{\infty}(\Omega)}=1.$
		On the other hand, since for any $s\in(0,1]$
	\[
		|a^s-b^s|\le |a-b|^s \quad \forall a,b\in[0,\infty),
	\]
	we have that $[u]_{\mathcal{W}^{s,\infty}(\mathbb{R}^{n+m})}\le\nicefrac{1}{R_s}.$
	Hence, using this functions as a test function in the variational problem 
	defining $\Lambda_\infty (s) $ we get 
	\begin{equation} \label{kkk}
		\Lambda_\infty (s) \leq \frac{1}{R_s}.
	\end{equation}

From \eqref{ppp} and \eqref{kkk} we obtain the desired result. 
\end{proof}

	\begin{lem} \label{lema.conv.unif.autov} Let $u_{p}$ be a positive 
	eigenfunction for $\lambda_{1}(s,p)$ normalized according to 
	$\|u_{p}\|_{L^p (\Omega)}=1$. Then, there exists a sequence 
	$p_j \to \infty$ such that
	\[
		u_j \to u
	\]
	uniformly in ${\mathbb{R}}^N$. 
	This limit function $u$ belongs to the space 
	$\mathcal{W}^{s,\infty}(\Omega)$ and is a solution to 
	the variational problem
	\begin{align*}
	\Lambda_\infty (s) 
		&
		= \min
		\left\{[u]_{\mathcal{W}^{s,\infty}(\Omega)}\colon 
		u\in \mathcal{W}^{s,\infty}(\Omega),  
		\|u\|_{L^\infty (\Omega)}=1, u=0 \mbox{ on }\partial\Omega
		\right\}.
	\end{align*}
	In addition, it holds that
	$$
	[\lambda_{1}(s,p)]^{1/p} \to \Lambda_\infty (s).
	$$
\end{lem}
	
	\begin{proof} 
	Let $\alpha>1$ and
	$$
		R_{s\alpha}= 
		\max_{(x,y) \in \Omega} \min_{(z,w) \in \partial \Omega}  
		(|x-z|^{s\alpha}+|y-w|^{s\alpha}).
	$$
	\noindent We first claim that
	\begin{equation}\label{eq:rsa}
			\dfrac{(R_s)^{\alpha}}{2^{\alpha-1}}\le R_{s\alpha}
	\end{equation}
	for any $\alpha>1.$ To this end, let $(x_0,y_0)\in \Omega$ such that
	\[
		R_s= \min_{(z,w) \in \partial \Omega}  (|x_0-z|^{s}+|y_0-w|^{s}).
	\]
	Then for all $(z,w)\in\partial \Omega$ we have
	\begin{align*}
		(R_s)^\alpha&\le\left(|x_0-z|^s+|y_0-w|^s\right)^\alpha
		\le 2^{\alpha-1}\left(|x_0-z|^{s\alpha}+|y_0-w|^{s\alpha}\right)\\
		&\le2^{\alpha-1}R_{s\alpha,}
	\end{align*}
	that is, \eqref{eq:rsa}.
	On the other hand, it is clear that if $s\alpha<1$ we have that
	\[
		u_\alpha(x,y)=\left(1-\dfrac{|x-x_0|^{\alpha s}+|y-y_0|^{\alpha s}}
		{R_{s\alpha}}\right)_+
	\]
	belongs to $\widetilde{\mathcal{W}}^{s,p}(\Omega)$ for all $p>1.$ Then
	\begin{equation}\label{eq:rsa1}
		(\lambda_{1}(s,p))^{1/p}\le\dfrac{[u_{\alpha}]_{\mathcal{W}^{s,p}
		(\mathbb{R}^{n+m})}}{\|u_{\alpha}\|_{L^p(\Omega)}}	
	\end{equation}
	for all $p>1$ and $1<\alpha<\nicefrac1s.$ Therefore
	\[
		\limsup_{p\to\infty}(\lambda_{1}(s,p))^{1/p}
		\le \dfrac{[u_{\alpha}]_{\mathcal{W}^{s,\infty}
		(\Omega)}}{\|u_{\alpha}\|_{L^\infty(\Omega)}} 
		\quad\forall \alpha\in(1,\nicefrac1s).
	\]
	Observe that if $\alpha\in(1,\nicefrac1s),$ by
	\eqref{eq:rsa}, we have
	\[
		\dfrac{|u_{\alpha}(x,y)-u_{\alpha}(z,y)|}{|x-z|^s}
		\le\dfrac{|x-z|^{(\alpha-1)s}}{R_{s\alpha}}
		\le 2^{\alpha-1}\dfrac{\text{diam}(\Omega)^{(\alpha-1)s}}
		{(R_s)^{\alpha}} 
	\]
	for all $(x,y)\neq(z,y)\in\overline{\Omega},$ and
	\[
		\dfrac{|u_{\alpha}(x,y)-u_{\alpha}(x,w)|}{|y-w|^s}
		\le\dfrac{|y-w|^{(\alpha-1)s}}{R_{s\alpha}}
		\le 2^{\alpha-1}\dfrac{\diam(\Omega)^{(\alpha-1)s}}
		{(R_s)^{\alpha}},
	\]
	for all $(x,y)\neq(z,y)\in\overline{\Omega},$ that is
	\[
		[u_{\alpha}]_{\mathcal{W}^{s,\infty}
		(\Omega)}\le 2^{\alpha-1}\dfrac{\diam(\Omega)^{(\alpha-1)s}}
		{(R_s)^{\alpha}}.
	\]
	Then, by
	\eqref{eq:rsa1} we get
	\[
		\limsup_{p\to\infty}(\lambda_{1}(s,p))^{\nicefrac1p}\le
		2^{\alpha-1}\dfrac{\diam(\Omega)^{(\alpha-1)s}}
		{(R_s)^{\alpha}} \qquad \alpha\in(1,\nicefrac1s),
	\]
	since $\|u_\alpha\|_{L^{\infty}(\Omega)}=1.$ Therefore,
	passing to the limit as $\alpha\to 1$ in the previous inequality
	we get
	\begin{equation}\label{eq:limt1}
		\limsup_{p\to\infty}(\lambda_{1}(s,p))^{
		\nicefrac1p}\le
		\dfrac{1}
		{R_s}=\Lambda_{\infty} (s).
	\end{equation}
	
	Our next goal is to show that
	\[
		\Lambda_{\infty} (s)\le
		\liminf_{p\to\infty}(\lambda_{1}(s,p))^{\nicefrac1{p}}.
	\]
	Let $p_j>1$ be such that
	\begin{equation}\label{eq:inf}
		\liminf_{p\to\infty}(\lambda_{1}(s,p))^{\nicefrac1{p}}
		=\lim_{j\to\infty}(\lambda_{1}(s,p_j))^{\nicefrac1{p_j}}.
	\end{equation}
	By \eqref{eq:limt1}, without of loss of generality, we can assume 
	\begin{equation}\label{eq:vie}
		(\lambda_{1}(s,p_j))^{\nicefrac1{p_j}}
		=[u_{p_j}]_{\mathcal{W}^{s,p_j}(\mathbb{R}^{n+m})}
		\le \Lambda_{\infty} (s) + \epsilon \qquad \forall j\in\mathbb{N},
	\end{equation}
	where $u_{p_j}$ is an eigenfunction for $\lambda_{1}(s,p_j)$ 
	normalized according to $\|u_{p_j}\|_{L^{p_j} (\Omega)}=1$ and $\epsilon$ is any positive number. 
	Then, by 
	Lemma \ref{lema:contincl}, we have that there exists a constant $C,$
	independent of $j,$ such that 
	\[
		|u_{p_j}|_{W^{s,p_j}(\Omega)}\le C\qquad \forall j\in\mathbb{N}.
	\]
	Therefore, for any $j\in\mathbb{N}$ there exists a constant 
	$C$ independent of $j,$
	such that
	\begin{equation}\label{eq:cota}
			\|u_{p_j}\|_{W^{s,p_j}(\Omega)}\le C.
	\end{equation}

	On the other hand, given $q>1$ such that $sq>2(n+m)$ 
	and taking  $t=s-\nicefrac{n+m}q,$  
	by H\"older's inequality, for any $p_j>q$ we 
	have that
	\[
		\|u_{p_j}\|_{L^q(\Omega)}^q\le |\Omega|^{1
		-\frac{q}{p_j}}\|u_{p_j}\|_{L^p(\Omega)}^q
		=|\Omega|^{1-\frac{q}{p_j}},
	\]
	and
	\begin{align*}
		|u_{p_j}|_{W^{t,q}(\Omega)}^q&=
		\int_{\Omega^2}\dfrac{|u_{p_j}(x,y)-
		u_{p_j}(z,w)|^q}{|(x,y)-(z,w)|^{sq}}\, dxdydzdw \\
		&\le |\Omega|^{2(1-\frac{q}{p_j})}	
		\left(\int_{\Omega^2}
		\dfrac{|u_{p_j}(x,y)-u_{p_j}(z,w)|^{p_j}}{|(x,y)-(z,w)|^{sp_j}}
		\, dxdydzdw\right)^{\frac{q}{p_j}}\\
		&\le|\Omega|^{2(1-\frac{q}{p_j})}\max\left\{1,\diam(\Omega)^{(n+m)
		\frac{q}{p_j}}\right\}	
		|u_{p_j}|_{W^{s,p_j}(\Omega)}^q.
	\end{align*}
	Hence, by \eqref{eq:cota}, for $j$ large there exists a constant
	$C,$ independent of $j,$ such that
	\[
		\|u_{p_j}\|_{W^{t,q}(\Omega)}\le C \max
		\left\{|\Omega|^{\frac1q-\frac1{p_j}},
		|\Omega|^{2(\frac1q-\frac1{p_j})},
		|\Omega|^{2(\frac1q-\frac1{p_j})}\diam(\Omega)^{\frac{n+m}{p_j}}
		\right\},
	\]
	that is, there exists $j_0>1$ such that $\{u_{p_j}\}_{j>j_0}$ is bounded in
	$W^{t,q}(\Omega).$ Then, since $tq>n+m,$ by Theorem \ref{teo:compacemb},
	there exists a subsequence $\{u_{k}\}_{k\in\mathbb{N}}$ of 
	$\{u_{p_j}\}_{j>j_0}$ and 
	a function $u\in C^{0,\gamma}(\overline{\Omega})$
	($0<\gamma<t-\nicefrac{(n+m)}q$) such that
	$
		u_k \to u
	$
	uniformly in $\overline{\Omega}.$
	
	Thus, if $q>1$ there exists $k_0\in\mathbb{N}$ such that
	$p_k>q$ if $k>k_0$ and therefore, by H\"older's inequality, for any
	$k>k_0$ we have
	\begin{align*}
		&\left(\int_{\Omega}\int_{\Omega_y}
		\dfrac{|u_k(x,y)-u_k(z,y)|^q}{|x-z|^{qs}} dzdxdy\right)^q \\
		&\le C^{\frac{1}{q}-\frac{1}{p_k}}
		\max\left\{1,\diam(\Omega)^{\frac{n}{p_k}}
		\right\}
		\left(
		\int_{\Omega}\int_{\Omega_y}
		\dfrac{|u_k(x,y)-u_k(z,y)|^{p_{k}}}{|x-z|^{p_ks+n}} dzdxdy
		\right)^\frac{1}{p_k}\\
		&\le C^{\frac{1}{q}-\frac{q}{p_k}}
		\max\left\{1,\diam(\Omega)^{\frac{n}{p_k}}
		\right\}[u_k]_{\mathcal{W}^{s,p_k}(\Omega)},
	\end{align*} 
	and similarly
	\begin{align*}
		\left(\int_{\Omega}\int_{\Omega_x}
		\dfrac{|u_k(x,y)-u_k(x,w)|^q}{|y-w|^{qs}} dwdxdy\right)^q \\ 
		\le C^{\frac{1}{q}-\frac{q}{p_k}}&
		\max\left\{1,\diam(\Omega)^{\frac{m}{p_k}}
		\right\}[u_k]_{\mathcal{W}^{s,p_k}(\Omega)}.
	\end{align*} 
	Here $C$ is a constant independent of $k$. Then passing to the limit as 
	$k\to\infty$ and using Fatou's lemma 
	we have that
	\begin{align*}
		\left(\int_{\Omega}\int_{\Omega_y}
		\dfrac{|u(x,y)-u(z,y)|^q}{|x-z|^{qs}} dzdxdy\right)^q
		&\le C^{\frac{1}{q}}\liminf_{k\to\infty}
		[u_k]_{\mathcal{W}^{s,p_k}(\Omega)}\\
		&\le  C^{\frac{1}{q}}\liminf_{p\to\infty}
		(\lambda_1(s,p))^{\nicefrac1p},\\
		\left(\int_{\Omega}\int_{\Omega_x}
		\dfrac{|u(x,y)-u(x,w)|^q}{|y-w|^{qs}} dwdxdy\right)^q
		&\le C^{\frac{1}{q}}\liminf_{k\to\infty}
		[u_k]_{\mathcal{W}^{s,p_k}(\Omega)}\\
		&\le  C^{\frac{1}{q}}\liminf_{p\to\infty}
		(\lambda_1(s,p))^{\nicefrac1p}
	\end{align*} 
	for all $q>1.$ Now passing to the limit as 
	$q\to\infty$ we obtain
	\begin{align*}
		\sup\left\{\dfrac{|u(x,y)-u(z,y)|}{|x-z|^{s}}\colon
		(x,y)\neq(z,y)\in\Omega\right\}&\le	\liminf_{p\to\infty}
		(\lambda_1(s,p))^{\nicefrac1p},\\
		\sup\left\{\dfrac{|u(x,y)-u(x,w)|}{|x-z|^{s}}\colon
		(x,y)\neq(x,w)\in\Omega\right\}&\le	\liminf_{p\to\infty}
		(\lambda_1(s,p))^{\nicefrac1p},
	\end{align*}
	that is
	\begin{equation}\label{eq:yaestamos}
		[u]_{\mathcal{W}^{s,\infty}(\Omega)}\le \liminf_{p\to\infty}
		(\lambda_1(s,p))^{\nicefrac1p}.
	\end{equation}
	
	To conclude we need to show that $\|u\|_{L^{\infty}(\Omega)}=1.$
	For all $q>1$ there exists $k_0\in\mathbb{N}$ such that
	$p_k>q$ if $k>k_0$ and therefore, by H\"older's inequality, 
	for any $k>k_0$ we get 
	\[
		\|u_{k}\|_{L^q(\Omega)}\le 
		|\Omega|^{\frac{1}{q}-\frac{1}{p_k}}\|u_{p_j}\|_{L^p(\Omega)}^q
		=|\Omega|^{\frac{1}{q}-\frac{1}{p_j}}.
	\]
	Then passing to the limit as 
	$k\to\infty$ and using that $u_k\to u$ uniformly in $\overline{\Omega},$
	$\|u\|_{L^q(\Omega)}\le1$ for all $q>1.$ Hence 
	$\|u\|_{L^{\infty}(\Omega)}\le 1.$
	On the other hand, for all $k$ we have 
	$1=\|u_k\|_{L^{p_k}(\Omega)}\le|\Omega|^{\nicefrac{1}{p_k}}
	\|u_k\|_{L^{\infty}(\Omega)}.$ Then, since $u_k\to u$ uniformly in 
	$\overline{\Omega},$ we get $1\le \|u\|_{L^{\infty}(\Omega)}.$
	Hence $\|u\|_{L^{\infty}(\Omega)}=1.$ 
	Thus, by \eqref{eq:yaestamos}, we get
	\[
		\Lambda_{\infty} (s)\le [u]_{\mathcal{W}^{s,\infty}(\Omega)}
		\le \liminf_{p\to\infty}
		(\lambda_1(s,p))^{\nicefrac1p},
	\]
	and by \eqref{eq:limt1} we conclude that
	\[
		\Lambda_{\infty} (s)=\lim_{p\to\infty}
		(\lambda_1(s,p))^{\nicefrac1p}.
	\]
	This ends the proof. 
	\end{proof}

	Using the geometric characterization given in Lemma \ref{lema.Lambda.infty} 
	we can compute $\Lambda_\infty(s)$ in some concrete examples.

	\begin{ejem} 
		When $\Omega = B_R$ is a ball of radius $R$ we have
		$$
		\Lambda_\infty (s) =  \frac{1}{R^s}.
		$$
	\end{ejem}

	\begin{ejem} 
		When $\Omega = (-R,R) \times (-L,L)$ is a rectangle in ${\mathbb{R}}^2$ 
		we have
		$$
			\Lambda_\infty (s) =  \frac{1}{\min \{R^s, L^s \} }.
		$$
	\end{ejem}

	\begin{remark}
		{\rm One can consider two different powers $r$ and $s$ in the definition of the pseudo $p-$Laplacian. In this case we get 
		that,
		$$
		\Lambda_\infty (r,s) = \max_{(x,y) \in \Omega} 
			\min_{(z,w) \in \partial \Omega}  (|x-z|^{r}+|y-w|^{s}).
		$$
		}
	\end{remark}

{\bf Viscosity solutions.}	
	To obtain an eigenvalue problem that is satisfied by the limit of the eigenfunctions $u_p$ when $p\to \infty,$ we need
	to introduce the definition of viscosity solutions. This is a notion of solution  different from the weak one considered before. We refer to \cite{CIL} for an introduction to the subject of viscosity solutions. In the theory of viscosity solutions the equation is evaluated for test functions at points where they touch the graph of a solution. Viscosity solutions are assumed to be continuous and the fractional Sobolev space is absent from the definition (no derivatives of a solutions are needed).
	
\begin{defi} \label{def.solucion.viscosa} (Viscosity solutions). Suppose that the function $u$ is continuous in ${\mathbb{R}}^{n+m}$ and that $u = 0$ in $\Omega^c$. We say that $u$ is a viscosity supersolution  of the equation
$
- \mathcal{L}_{s,p}u + \lambda |u|^{p-2}u =0
$
if the following holds: whenever $x_0 \in \Omega$ and $\varphi \in C_0^1({\mathbb{R}}^{n+m})$ (the test function) are such that
$\varphi(x_0) = u(x_0)$ and $\varphi (x) \leq u(x)$ for every $x \in {\mathbb{R}}^{n+m}$, then we have
$$
- \mathcal{L}_{s,p}\varphi (x_0) + \lambda |\varphi (x_0)|^{p-2}\varphi(x_0) \leq 0.
$$
The requirement for being a viscosity subsolution is symmetric: the test function is touching from above and the inequality is reversed. 

Finally, a viscosity solution is defined as being both a viscosity supersolution and a viscosity subsolution.
\end{defi}

For our eigenvalue problem, we have that a continuos weak solution is a viscosity solution. For the proof we refer to \cite{LL}.

	\begin{teo}
	An eigenfunction 
	$u \in C(\overline{\Omega})$ 
	(in the weak sense) is a
viscosity solution of the equation
$
- \mathcal{L}_{s,p}u + \lambda |u|^{p-2}u =0
$
in the sense of Definition \ref{def.solucion.viscosa}.
	\end{teo}
	
	We will also use the following lemmas.
	
	\begin{lem} \label{lema5.6} Assume that
		\begin{align*}
			&\left(A_p\right)^{\nicefrac1p}\to 
			A,
			&\left(B_p\right)^{\nicefrac1p}\to -B,\\
			&\left(C_p\right)^{\nicefrac1p}\to C,
			&\left(D_p \right)^{\nicefrac1p}\to -D,
		\end{align*}
		and that 
		$$
		\theta_p \to \Theta,
		$$
		as $p\to \infty$.
		If
		$$
		2^{1/p} (A_p + C_p)^{1/p} \geq (B_p +D_p + \theta_p^{p-1})^{1/p}
		$$
		for every $p$ large enough, then, passing to the limit, it holds that
		$$
		\max \{ A; C\} \geq \max \{ -B; -D; \Theta \}.
		$$ 
	\end{lem}
	
	\begin{proof} First, assume that $A > C $ and $-B > \max \{-D; \Theta \}$. Then for $p$ large enough
	we have 
	$A_p \geq  C_p $, $-B_p \geq -D_p$ and $-B_p \geq (\theta_p)^{p}$. Then 
	taking $p\to \infty$ in 
	$$
		(A_p)^{\nicefrac1p} 
		2^{\nicefrac1p} \left(1 + \frac{C_p}{A_p}\right)^{\nicefrac1p} 
		\geq (B_p)^{\nicefrac1p}
		\left(1 +\frac{D_p}{B_p} + \frac{\theta_p^{p-1}}{B_p}\right)^{\nicefrac1p}
		$$ 
		we get
		$$
		A \geq -B.
		$$
	The rest of the cases ($A=C$, $A<C$, etc) can be handled in an analogous way.
	\end{proof}
	
	\begin{lem} For a smooth test function $\phi$ let
	$$
		A_p=\int_{\mathbb{R}^n}
			\dfrac{|\phi(x_p,y_p)-\phi(z,y_p)|^{p-2}(\phi(x_p,y_p)
			-\phi(z,y_p))^+}{|x_p-z|^{n+sp}}dz.
			$$
			If $x_p\to x_0$, $y_p \to y_0$ as 
			$p\to \infty$, then
			$$
			(A_p)^{\nicefrac1p } \to A = \sup_z \frac{\phi(x_0,y_0)
			-\phi(z,y_0)}{|x_0-z|^{s}}.
			$$
	\end{lem}
	
	\begin{proof}
	We just have to observe that 
	$$
	(A_p)^{\nicefrac1p}= \left(\int_{\mathbb{R}^n}
			\dfrac{|\phi(x_p,y_p)-\phi(z,y_p)|^{p-2}(\phi(x_p,y_p)
			-\phi(z,y_p))^+}{|x_p-z|^{n+sp}}dz \right)^{1/p}.
	$$
	The integrand satisfies
	$$
	\begin{array}{l}
	\dfrac{|\phi(x_p,y_p)-\phi(z,y_p)|^{p-2}(\phi(x_p,y_p)
			-\phi(z,y_p))^+}{|x_p-z|^{n+sp}} \\
			\qquad \sim 
			\dfrac{|\phi(x_0,y_0)-\phi(z,y_0)|^{p-2}(\phi(x_0,y_0)
			-\phi(z,y_0))^+}{|x_0-z|^{n+sp}}
			\end{array}
	$$
	and hence the result follows from the fact that $
	\left(\int f^p\right)^{\nicefrac1p} \to \| f \|_\infty$.
	\end{proof}

	\begin{lem} \label{lema.sol.viscosa}
		Any uniform limit of $u_{p}$ a sequence of eigenfunctions for $\lambda_{1}(s,p)$ 
		normalized according to $\|u_{p}\|_{L^p (\Omega)}=1$, $u$ is a 	
		nontrivial solution to 	
		\[
			\begin{cases}
				\max\{A; C\}= \max\{-B;-D; \Lambda_\infty (s) u\}
 				& \mbox{ in } \Omega, \\
 			 	u=0 & \mbox{ in }\Omega^c,\\
			\end{cases}
		\]
		in the viscosity sense. Here
		\begin{align*}
			&A= \sup_w \frac{u(x,w)-u(x,y)}{|y-w|^{s}},
			&B= \inf_w \frac{u(x,w)- u(x,y)}{|y-w|^{s}},\\
			&C= \sup_z \frac{u (z,y)- u(x,y)}{|x-z|^{s}},
			&D= \inf_z \frac{u(z,y)- u(x,y)}{|x-z|^{s}}.
		\end{align*}
	\end{lem}

	\begin{proof} 
		We call $u_p$ a sequence of solutions to 
		$
		- \mathcal{L}_{s,p}u + \lambda |u|^{p-2}u =0
		$
		that converges uniformly to $u$.
		That $u=0$ in $\Omega^c$ follows since 
		$u_p =0$ in $\Omega^c$ and we have uniform 
		convergence. 

		Let $\phi\in C^1_0(\mathbb{R}^{n+m})$ be such that 
		$u-\phi$ has a strict minimum at $(x_0,y_0)\in \Omega$.
		Since $u_p$ converges uniformly to $u$ we have that there exist 
		$(x_p,y_p)\in \Omega$ such that $u_p-\phi$ has a minimum at $(x_p,y_p)$ 
		and $(x_p,y_p) \to (x_0,y_0)$ as $p\to \infty$. Since $u_p$ is a 
		viscosity solution to
		$-\mathcal{L}_{s,p}v(x,y)+\lambda_1(s,p)v(x,y)^{p-1}=0$ in 
		$\Omega,$ we obtain
		\begin{equation}\label{eq:ultimaeq}
			\begin{aligned}
			((\lambda_{1}(s,&p))^{\nicefrac{1}{(p-1)}} 
			u_p (x_p,y_p) )^{p-1} \leq \\
			\le&2 \int_{\mathbb{R}^n}
			\dfrac{|\phi(x_p,y_p)-\phi(z,y_p)|^{p-2}(\phi(x_p,y_p)
			-\phi(z,y_p)
			)}{|x_p-z|^{n+sp}}dz\\
			& + 2\int_{\mathbb{R}^m}
			\dfrac{|\phi(x_p,y_p)-\phi(x_p,w)|^{p-2}(\phi(x_p,y_p)-\phi(x_p,w)
			)}{|y_p-w|^{m+sp}}dw\\
			=&2(A_p-B_p+C_p-D_p),
		\end{aligned}
		\end{equation}
		where
		\begin{align*}
			&A_p=\int_{\mathbb{R}^n}
			\dfrac{|\phi(x_p,y_p)-\phi(z,y_p)|^{p-2}(\phi(x_p,y_p)
			-\phi(z,y_p))^+}{|x_p-z|^{n+sp}}dz, \\
			&B_p=\int_{\mathbb{R}^n}
			\dfrac{|\phi(x_p,y_p)-\phi(z,y_p)|^{p-2}(\phi(x_p,y_p)
			-\phi(z,y_p))^-}{|x_p-z|^{n+sp}}dz,\\
			&C_p=\int_{ \mathbb{R}^m}
			\dfrac{|\phi(x_p,y_p)-\phi(x_p,w)|^{p-2}(\phi(x_p,y_p)-
			\phi(x_p,w)
			)^+}{|y_p-w|^{m+sp}}dw ,\\
			&D_p=\int_{ \mathbb{R}^m}
			\dfrac{|\phi(x_p,y_p)-
			\phi(x_p,w)|^{p-2}(\phi(x_p,y_p)-\phi(x_p,w)
			)^-}{|y_p-w|^{m+sp}}dw .
		\end{align*}
		We observe that
		\begin{align*}
			&\left(A_p\right)^{\nicefrac1p}\to 
			A,
			&\left(B_p\right)^{\nicefrac1p}\to -B,\\
			&\left(C_p\right)^{\nicefrac1p}\to C,
			&\left(D_p \right)^{\nicefrac1p}\to -D,
		\end{align*}
		and
		$$
			(\lambda_{1}(s,p))^{\nicefrac1{(p-1)}} u_p (x_p,y_p)  
			\to \Lambda_\infty u (x_0,y_0).
		$$
		Hence, taking limit as $p\to \infty$ in \eqref{eq:ultimaeq}, from Lemma \ref{lema5.6}, we get
		$$
			 \max\{-B;- D; \Lambda_\infty (s) u (x_0,y_0)\}\le\max\{A; C\}.
		$$

		Now, if $\psi$ is such that $u-\psi$ has a strict minimum at $(x_0,y_0)
		\in \Omega$.
		Since $u_p$ converges uniformly to $u$ we have that there exist 
		$(x_p,y_p)\in \Omega$ such that $u_p-\psi$ has a minimum at $(x_p,y_p)$ 
		and $(x_p,y_p) \to (x_0,y_0)$ as $p\to \infty$. Since $u_p$ is a 
		solution to $-\mathcal{L}_{s,p}v(x,y)+\lambda v(x,y)^{p-1}=0$ 
		in $\Omega$ we obtain
		\begin{align*}
			((\lambda_{1,p}&)^{\nicefrac1{(p-1)}} u_p (x_p,y_p) )^{p-1} \geq \\
			&\ge 2\int_{\mathbb{R}^n}
			\dfrac{|\psi(x_p,y_p)-\psi(z,y_p)|^{p-2}(\psi(x_p,y_p)-\psi(z,y_p)
			)}{|x_p-z|^{n+sp}}dz\\
			&\quad +2\int_{\mathbb{R}^m}
			\dfrac{|\psi(x_p,y_p)-\psi(x_p,w)|^{p-2}(\psi(x_p,y_p)-\psi(x_p,w)
			)}{|y_p-w|^{m+sp}}dw,
		\end{align*}
		and, arguing as before, we obtain
		$$
			\max\{A; C\}\geq \max\{-B;-D; \Lambda_\infty (s) u (x_0,y_0) \}.
		$$
		\end{proof}


\end{document}